\documentclass{article}
\usepackage{amssymb,amsmath,amsthm,graphicx, color, amsfonts}
\usepackage[english]{babel}
\usepackage{authblk}
\usepackage{geometry}
\def\supp{\mathop{\rm supp}\nolimits}
\newtheorem{theorem}{Theorem}

\newtheorem{lemma}[theorem]{Lemma}
\newtheorem{proposition}[theorem]{Proposition}

\newtheorem{definition}[theorem]{Definition}
\newtheorem{remark}[theorem]{Remark}
\newtheorem{example}[theorem]{Example}

\renewenvironment{proof}[1][.]{%
\bigskip\noindent{\bf Proof#1 }}{%
\hfill$\blacksquare$\bigskip}
\begin{document}
\pagestyle{myheadings}
\title{The Ergodic Theorem for a new kind of attractor of a GIFS.}
\author[1]{Elismar R. Oliveira
\thanks{
Instituto de Matem\'atica e Estat\'istica - UFRGS\\
Av. Bento Gon\c calves, 9500\\
Porto Alegre - 91500 - RS -Brasil\\
Email: elismar.oliveira@ufrgs.br}
}
\affil[1]{Universidade Federal do Rio Grande do Sul\\
}
\date{\today}
\maketitle

%\boldmath{ %retirar depois
\noindent\rule{\textwidth}{0.4mm}
\begin{abstract} In 1987, J. H. Elton~\cite{MR922361}, has proved the first fundamental result in convergence of IFS, the Elton's Ergodic Theorem. In this work we prove the natural  extension of this theorem to the projected Hutchinson measure $\mu_{\alpha}$ associated to a GIFSpdp $\mathcal{S}=\left(X, (\phi_j:X^{m} \to X)_{j=0,1, ..., n-1}, (p_j)_{j=0,1, ..., n-1}\right),$ in a compact metric space $(X,d)$. More precisely, the average along of the trajectories $x_{n}(a)$ of the GIFS, starting in any initial points $x_0, ..., x_{m-1} \in X$ satisfies, for any $f \in C(X , \mathbb{R})$,
$$\lim_{N\to +\infty} \frac{1}{N}\sum_{n=0 }^{N-1}   f(x_{n}(a)) = \int_{X} f(t) d\mu_{\alpha}(t),$$
for almost all $a \in \Omega=\{0,1, ..., n-1\}^{\mathbb{N}}$, the symbolic space. Additionally, we give some examples and applications to Chaos Games and Nonautonomous Dynamical Systems defined by finite difference equations.
\end{abstract}
\noindent\rule{\textwidth}{0.4mm}\\

\vspace {.3cm}
\tableofcontents

%\noindent\rule{\textwidth}{0.4mm}

\section*{Introduction}

In 2008, A. Mihail and R. Miculescu \cite{MR2415407}, has introduced the Generalized Iterated Function Systems (GIFS, for short). They prove that there exist a fractal attractor and give estimates of the rate of convergence for contractive GIFS. In 2009, Alexandru Mihail \cite{MR2568892}, has considered the Hutchinson measure associated to a GIFS with place dependent probabilities (GIFSpdp for short) that generalizes the classical Hutchinson measure, the invariant measure, associated to an Iterated Function System (IFS for short).

In this work, the central idea is to extend the GIFS to an IFS on a bigger space, that we call \emph{Extended GIFS}. From this IFS we get an \emph{Extended Hutchinson measure}. In the rest of the paper, we set up the properties of the extended Hutchinson measure and his relation with the original GIFS.

Using the extension, we prove an ergodic theorem that extends the classic Elton's ergodic theorem for IFS, Elton~\cite{MR922361}, to the Hutchinson measure associated to a GIFS. From our results, we prove a  Chaos game theorem for GIFS. As an application, we get some results on the stability and the asymptotic behavior of nonautonomous dynamical systems defined by finite difference equations. Also, we show haw to recover some properties of Gibbs measures for H\"older potentials through an appropriate GIFS that is builded from an expansive endomorphism.

The paper has three sections. In the Section~\ref{background} we recall the basic facts about GIFS and GIFSpdp, that we will use in the rest of the paper.  In  the Section~\ref{dyn point view ergodic} we introduce the extension of a GIFS and prove the Ergodic Theorem (Theorem~\ref{EltErgTheoremGIFS}). Section~\ref{applications}  is devoted to applications. The main goal here is to prove the Chaos Game Theorem (Theorem~\ref{chaosgametheorem}), allowing us to draw the attractor of the extended GIFS and his projection.

We believe that the tools we develop will be very useful to forthcoming works and for other researchers in this area. The ergodic theorem we prove, and his consequences, represents a real advance in the understand of GIFS.

\section{Background on GIFS and GIFSpdp}\label{background}
In this section, we will recall the basic definitions and results on the theory of GIFS. See R. Miculescu~\cite{MR3180942}, for more details and the notation. We notice that the word ``generalized" has been used in several different ways representing more general components of a classical IFS. Here, generalized, means that functions are from $X^m$ to $X$ instead $X$ to $X$. \\

\subsection{Generalized iteration function system (GIFS)}
Let $(X,d)$ be a compact~\footnote{We assume compactness to avoid technicalities. Many of the results that we present here are true if $(X,d)$ is just complete. It is sufficient because the measures we use are always supported in the attractors that are compact sets.} metric space (typically $X=[0,1]$,  $X=\{0,1,..., d-1\}^{\mathbb{N}}$, $X=[0,1]^{\mathbb{N}}$, etc).  Consider the topology  on $X^{m}$ given by
$$\displaystyle d_{\infty}( (x_1, x_2, ..., x_m), (y_1, y_2, ..., y_m) ) = \max_{i=1,...,m} d(x_i , y_i),$$
then $(X^{m},d_{\infty}) $ is also a compact metric space.

\begin{definition}
A (continuous) \textbf{generalized iterated function system} (GIFS) of degree $m$ is a (finite) family $\mathcal{S}$ of continuous functions $\phi_j: X^{m} \to X$, denoted $\mathcal{S}=(X, (\phi_j)_{j=0...n-1}).$
\end{definition}
See Secelean \cite{MR2932770, MR3198614} for the analogous theory for countable GIFS.
\textbf{\boldmath{In order to avoid technicalities we assume that $m=n=2$ that is, two maps in $X^2$ (see Remark~\ref{simplify m=n=2} for additional details). So the standard GIFS is $\mathcal{S}=(X, (\phi_j)_{j=0,1})$ where $  \phi_{0} , \phi_{1} : X^{2} \to X.$}}
We recall that,
$$Lip(X^2, X)=\{f\; |\; d(f(x_0, y_0), f(x_1, y_1))\leq C d((x_0, y_0), (x_1, y_1)), \; C:=Lip(f)\}$$ and
$$Lip_{a,b}(X^2, X)=\{f \; |\; d(f(x_0, y_0), f(x_1, y_1))\leq a d(x_0, x_1)+ b d(y_0, y_1), \; a,b >0\}.$$
From now on we will assume the contraction hypothesis:\\

\textbf{E1 - } Each $\phi_j: X^{2} \to X$ is in $Lip_{a_{j},b_{j}}(X^2, X)$ and $a_{j}+b_{j} <1$. In particular, all the $\phi_j$ are Lipschitz contractions and $Lip(\phi_j)=a_{j}+b_{j}$.\\

As usual, we denote $\mathcal{K}(X) \subset  \mathcal{P}(X)=2^X$, the family of compact subsets of $X$. Moreover, $\mathcal{K}^{*}(X)=\mathcal{K}(X) \setminus \{ \varnothing\}$.
\begin{definition} Given $f: X^{2} \to X$ we define the associated set function $F_{f}: \mathcal{P}^{*}(X)^{2} \to \mathcal{P}^{*}(X)$ by
$F_{f}( K_1,  K_2)= f(K_1\times  K_2).$
Also, we define the function $F_{\mathcal{S}}: \mathcal{K}^{*}(X)^{2} \to \mathcal{K}^{*}(X)$ associated to $\mathcal{S}$ by
$F_{\mathcal{S}}( K_1,  K_2)= \bigcup_{j=0,1} F_{\phi_{j}}( K_1, K_2).$ A set $Y\subseteq X$ is self-similar(or fractal) with respect to $\mathcal{S}$ if $F_{\mathcal{S}}( Y,  Y)= Y.$
\end{definition}

The map $F_{\mathcal{S}}$ is sometimes called, Fractal operator, Barnsley's Function or Hutchinson's operator, in the literature. From Mihail and Miculescu~\cite{MR2415407}, Theorem 3.5,  we know that under the hypothesis E1, there exists a unique attractor $A(\mathcal{S}) \in \mathcal{K}^{*}(X)$ for the GIFS that depends continuously on $\phi_j$. That is, $A(\mathcal{S})$ is self-similar ($ F_{\mathcal{S}}(A(\mathcal{S}), A(\mathcal{S}))= A(\mathcal{S})$)
and,  for any $H_0,  H_{1} \in \mathcal{K}^{*}(X)$ the recursive sequence of compact subsets $H_{j+2}= F_{\mathcal{S}}( H_{j+1},  H_{j}),$
converges to $A(\mathcal{S})$ with respect to the Hausdorff metric:
$A(\mathcal{S}) = \displaystyle\lim_{j \to \infty}  H_{j}.$

The natural question to make about GIFS is if they offer some new fractals. The positive answer is given by Mihail and Miculescu~\cite{MR2415407} through examples and in the recent work Strobin~\cite{MR3263451} for a more general case. We will discuss that in the end of the Section~\ref{extsection}. We should mention that recently,  in 2015, Dumitru, Ioana, Sfetcu and Strobin~\cite{MR3300886}  has considered  many questions regarding to the extension of the concept of GIFS for topological contractions assuming that the family of maps is not just finite or countable but possibly an arbitrary family $\mathcal{F}$ of maps from $X^m$ to $X$, satisfying suitable hypothesis. Several results ar obtained by using code spaces (see \cite{MR0000000} for details).

\subsection{GIFS with place dependent probabilities (GIFSpdp)}
In this section, we use the notation in R. Miculescu~\cite{MR3180942}. The set $Prob(X)$ will always be the set of regular Borel probabilities on $X$ with respect to the Borel sigma algebra induced by the metric.

\begin{definition}
A \textbf{generalized iteration function system with place dependent probabilities} (GIFSpdp) is a family $\mathcal{S}$ of continuous functions $\phi_j: X^{2} \to X$,  and weight functions (probabilities) $p_j: X^{2} \to [0,1]$ such that  $p_0(x,y) +  p_{1}(x,y) =1$, denoted $\mathcal{S}=(X, (\phi_j)_{j=0,1}, (p_j)_{j=0,1}).$
\end{definition}

One special case is when  the probabilities are given by a potential function $u: X \to \mathbb{R}$, then $p_j(x,y)= u(\phi_{j}(x,y))$ and $u(\phi_{0}(x,y))+u(\phi_{1}(x,y))=1$. We denote such case as a \textbf{uniform} GIFSpdp according to Lopes and Oliveira~\cite{MR2461833}.\\

\noindent\textbf{E2 - } For a GIFSpdp we  assume two hypothesis on the weights:
\begin{enumerate}
  \item[a)] Any $p_i(x,y) \geq \delta >0$ for any $i=0,1,\; x, y \in X$;
  \item[b)] Any $p_i(x,y)$ is in $Lip_{c_{i}, d_{i}}(X^2, [0,1])$ with $c_{i}+ d_{i}<1$.
\end{enumerate}

We recall that $p_i(x,y)$ is Dini continuous if  $\int_{0}^{\varepsilon} \frac{Q_{i}(t)}{t} dt < \infty$ for some $\varepsilon>0$, where $Q_{i}$ is the modulus of continuity of $p_i$,
$$|p_i(x,y)- p_i(x',y')| \leq Q_{i}\left(d((x,y), (x',y'))\right), \; \forall (x,y)\neq  (x',y').$$
For instance if $p_i$ is $\beta$-H\"older ($Q_{i}(t)= k t^{\beta}$) or $p_i$ is $k$-Lipschitz ($Q_{i}(t)= k \, t$) then $p_i(x,y)$ is Dini continuous.

\begin{definition} Given $\mathcal{S}=\left(X, (\phi_j)_{j=0,1}, (p_j)_{j=0,1}\right)$ we define (see R. Miculescu \cite{MR3180942}), the transference operator $B_{\mathcal{S}}: C(X , \mathbb{R}) \to C(X^2, \mathbb{R})$  by
$$ B_{\mathcal{S}}(f) (x,y)= \sum_{j=0,1} p_{j}(x,y) f(\phi_{j}(x,y)),$$
for all $(x,y) \in X^2$. And the Markov operator $ \mathcal{L}_{\mathcal{S}} : Prob(X) \times Prob(X) \to Prob(X)$ by
$$\int_{X} f(t) d\mathcal{L}_{\mathcal{S}} (\mu, \nu)(t) = \int_{X^{2}} B_{\mathcal{S}}(f) (x,y) d(\mu\times \nu)(x,y),$$
for any $\mu, \nu \in Prob(X)$ and any continuous $f: X \to \mathbb{R}$.
\end{definition}

Under the hypothesis E1 and E2 we get, from R. Miculescu \cite{MR3180942}, Theorem 4.4, that:\\

1- There is a unique $\mu_{\mathcal{S}} \in Prob(X)$ such that $\mathcal{L}_{\mathcal{S}} (\mu_{\mathcal{S}}, \mu_{\mathcal{S}})=\mu_{\mathcal{S}}$;\\

2- $\supp(\mu_{\mathcal{S}})=A(\mathcal{S})$, the attractor of the GIFS;\\

3- For any $\mu_0, \mu_1 \in Prob(X)$  the sequence $\mu_{j+2} = \mathcal{L}_{\mathcal{S}} (\mu_{j}, \mu_{j+1})$ converges in the Monge-Kantorovich distance $d_{H}$~\footnote{ $\displaystyle d_{H}(\mu,\nu)=\sup_{Lips(f)\leq 1} \left|\int f d \mu - \int f d\nu \right| $, for any $\mu, \nu \in Prob(X)$.} (see \cite{MR3014680}, Definition 2.53),  to $\mu_{\mathcal{S}} $.\\
\begin{definition}
The Hutchinson measure $\mu_{\mathcal{S}}$ associated to a GIFSpdp is the unique solution of  $\mathcal{L}_{\mathcal{S}} (\mu_{\mathcal{S}}, \mu_{\mathcal{S}})=\mu_{\mathcal{S}}$.
\end{definition}
In the next sections we will consider the fixed point of another operator  to get the \textbf{extended Hutchinson measure} in $X^2$.

\section{Dynamical point of view: the ergodic theorem} \label{dyn point view ergodic}
A GIFS is not a typical discrete dynamical system because $\phi_i: X^{2} \to X$, is not an endomorphism from $X^{2} \to X^{2}$. However we can consider the dynamics of an IFS in $X^{2}$,  whose projection in the first coordinate, is the orbit of the GIFS. In this section, we  assume the hypothesis E1 and E2.

\subsection{The extension of a GIFS to an IFS}\label{extsection}

In order to analyze the orbits of a GIFS one can to embed  $\mathcal{S}=(X, (\phi_j)_{j=0,1})$ in to an IFS  $\hat{\mathcal{S}}=\left(X^{2},  (\hat{\phi}_i (x,y))_{i=0,1}\right)$ where $\hat{\phi}_i : X^{2} \to X^{2},$ is given by $\hat{\phi}_i (x,y) = (y, \phi_i(x,y))$  that is  \textbf{the extension} of $\mathcal{S}.$

\begin{remark} \label{simplify m=n=2}
\textbf{\boldmath{We point out that, make a extension of  a GIFS in $X^{2}$ instead $X^m$, is not actually a restriction. If we consider $(X,d)$, $m, n \geq 2$ (degree $m$ and $n$ maps) and a general GIFS  as a family $\mathcal{S}$ of continuous functions $\phi_j: X^{m} \to X$, denoted $\mathcal{S}=(X, (\phi_j)_{j=0...n-1})$, then his extension will be the  IFS  $\hat{\mathcal{S}}=\left(X^{m},  (\hat{\phi}_i (x))_{i=0...n-1}\right)$ where $\hat{\phi}_i : X^{m} \to X^{m},$ is given by $\hat{\phi}_i (x) = (\theta(x), \,\phi_i(x))$, where $x=(x_{ 0}, ...,x_{ m-2}, x_{m-1})$ and $\theta: X^{m} \to X^{m-1}$ is given by $\theta(x_{ 0}, ...,x_{ m-2}, x_{m-1})=(x_{1}, ...,x_{ m-2})$. Additionally, we produce the orbits by choosing sequences $a=(a_{0}, a_{1}, ....) \in\{0,1, ... n-1\}^{\mathbb{N}}$, which will make de proofs enormously hard to read. So, in the rest of the paper we will make the proofs for $m=n=2$ making easier to recognize the key elements in the demonstrations~\footnote{That is, $\theta(x_{ 0}, x_{1})=x_{1}$ and $\hat{\phi}_i (x_{ 0}, x_{1}) = (\theta(x), \,\phi_i(x))= ( x_{1}, \,\phi_i(x_{ 0}, x_{1}))$, for $i=0,1$.}.}}
\end{remark}

We want to investigate the relation between the dynamics of this  IFS and the properties of the GIFS and its Hutchinson measure.

\begin{definition}\label{GIFSorbit}
Given GIFS $\mathcal{S}=(X, (\phi_j)_{j=0,1})$ and $a=(a_{0}, a_{1}, ....) \in\{0,1\}^{\mathbb{N}}$ a fixed sequence, the orbit $x_{0}, x_{1} \in X$   is the sequence obtained by the nonautonomous recurrence relation $x_{0}(a) = Z_{0, a} (x_0, x_1)=x_{0}$, $x_{1}(a)= Z_{1, a} (x_0, x_1)=x_{1} $ and  $x_{j}(a) = Z_{j, a} (x_0, x_1)$ (or $x_{j} = Z_{j, a}$ for short) for $j\geq 2$ where $Z_{j+1, a} =  \phi_{a_{j-1}}(Z_{j -1, a}, Z_{j, a}), \; j \geq 1.$
\end{definition}

The iterations by  a sequence $a=(a_{0}, a_{1}, ....) \in \Omega=\{0,1\}^{\mathbb{N}}$ defines the behavior of an IFS :
$$\{\hat{\phi}_{a_{0}} (x,y), \,  \hat{\phi}_{a_{1}}(\hat{\phi}_{a_{0}} (x,y)), \, ...\},$$
that is
$\hat{\phi}_{a_{j-1}}\circ\cdots\circ\hat{\phi}_{a_{1}}\circ\hat{\phi}_{a_{0}} (x,y).$
If we take two compact sets $H_0, H_1 \in \mathcal{K}^{*}(X)$ and
$x_i \in H_i, \; i=0,1$ then
  $$x_{2} (a) = \phi_{a_0} (x_0(a), x_1(a)) \in H_{2} =F_{\mathcal{S}}( H_{0},  H_{1}),$$
  $$x_{3} (a) = \phi_{a_1} (x_1(a), x_{2}(a) ) \in H_{3} = F_{\mathcal{S}}( H_{1},  H_{2}), {\rm etc}.$$
\begin{figure}[h!]
  \centering
  \includegraphics[width=6cm]{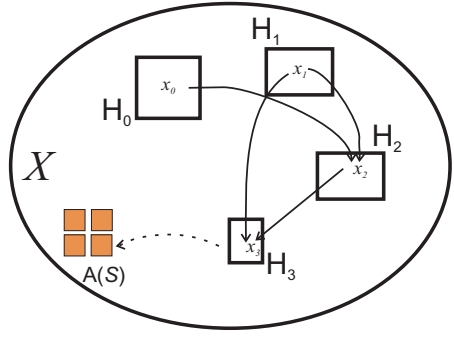}\\
  \caption{Dynamics of $\mathcal{S}$, the orbit of $x_0, x_1$ by $a$. }\label{dyntyp}
\end{figure}

If we define $Z_{0, a} (x_0, x_1)= x_{0} \in H_{0}$, $Z_{1, a} (x_0, x_1)= x_{1}\in H_{1}$ and $Z_{m, a} (x_0, x_1)= x_{m} (a) \in H_{m},$ for $m \geq 2$, then any accumulation point of this sequence will be in $A(\mathcal{S})$, see Figure~\ref{dyntyp}. Thus, this sequences are significant on the asymptotic behavior of the GIFS.

\begin{remark} The orbit by  a GIFS $\mathcal{S}=\left(X, (\phi_j)_{j=0,1}\right)$ is the projection on the first coordinate of the orbit $\{(x_{j}, x_{j+1}), \; j \geq 0\}$ by the IFS associated $\hat{\mathcal{S}}$:
\begin{align*}
  \hat{\phi}_{a_{0}} (x_0, x_1) &= (x_1, \phi_{a_{0}}(x_0, x_1)) & \text{ or } & x_{2}=Z_{2, a} (x_0, x_1)= \phi_{a_{0}}(x_0, x_1), \\
  \hat{\phi}_{a_{1}} (x_1, x_2) &= (x_2, \phi_{a_{1}}(x_1, x_2)) &  \text{ or } & x_{3}=Z_{3, a} (x_0, x_1)= \phi_{a_{1}}(x_1, x_2,)
\end{align*}
and so on, that is
$\hat{\phi}_{a_{j-1}}\circ\cdots\circ\hat{\phi}_{a_{1}}\circ\hat{\phi}_{a_{0}} (x_0, x_1)
= (x_{j}, x_{j+1}), \; j \geq 1.$
\end{remark}

In Mihail and Miculescu \cite{MR2415407}, Ex. 4.3, we found an example of a GIFS $\mathcal{S}=\left(X, (\phi_j)_{j=0,1}\right)$ whose attractor have infinite Hausdorff dimension, thus it is not an attractor of any finite Lipschitz IFS. This means that the GIFS theory gives us new fractal sets. In the recent work of Strobin (see Strobin~\cite{MR3263451}, Theorem 6), the author proves that certain GIFS in $X^m$ for $X \subset \mathbb{R}^{2}$ formed by generalized Matkowski contractions has attractors that are not attractors for any GIFS in $X^r$ with $1\leq r \leq m-1$ in particular for $r=1$ when the GIFS to become a classical IFS in $X$. However, the Proposition~\ref{doubleattractor} shows that this new fractals $A(\mathcal{S})$ contains the projection of  $A(\hat{\mathcal{S}})$, the attractor of the extension. The IFS $\hat{\mathcal{S}}$ is not contractive, thus we need to consider higher powers  to prove the there exists an attractor.

\begin{definition}\label{Higher powers}
Given  an IFS $\mathcal{R}=\left(Z,  (f_i (z))_{i=0,1}\right)$ we define his kth power  as the IFS $\mathcal{R}^{k}=\left(Z,  (f_{i_{1}...i_{k}} (z))_{i=0,1}\right)$ where $f_{i_{1}...i_{k}} (z)= f_{i_{k}}  \circ\cdots \circ f_{i_{1}} (z)$. We said that $\mathcal{R}$  is eventually contractive if $\mathcal{R}^{k}$ is contractive for some $k\geq 1$. Analogously, if $\mathcal{R}=\left(Z,  (f_i (z))_{i=0,1}, (p_i (z))_{i=0,1}\right)$ is a IFSpdp we define his kth-power  as the IFSpdp $\mathcal{R}^{k}$ where $p_{i_{1}...i_{k}} (z)= p_{i_{k}} (f_{i_{k}}  \cdot \cdots \cdot f_{i_{1}} (z))  \circ\cdots \circ p_{i_{1}} (z)$.
\end{definition}

\begin{lemma}\label{eventually contract 2} The IFS $\hat{\mathcal{S}}^2$ is contractive, more precisely  $Lip(\phi_{j})=\lambda =\max_{j} [a_{j}+b_{j} ]<1$ for any $j=0,1$. In particular $\hat{\mathcal{S}}$ has an attractor $A(\hat{\mathcal{S}})=A(\hat{\mathcal{S}}^2)$ \textbf{\boldmath{ (Or, $\hat{\mathcal{S}}^m$ is contractive and $A(\hat{\mathcal{S}})=A(\hat{\mathcal{S}}^m)$, if $\mathcal{S}$ has degree $m$.)}} and $\hat{H}_{n+1}=F_{\hat{\mathcal{S}}}(\hat{H}_n)\to A(\hat{\mathcal{S}})$ for any  $\hat{H}_0 \in \mathcal{K}^{*}(X^{2})$. \end{lemma}
\begin{proof} We know that each $\phi_j: X^{2} \to X$ is in $Lip_{a_{j},b_{j}}(X^2, X)$ and $a_{j}+b_{j} <1$, that is,
$d(\phi_{j}(x,y), \phi_{j}(x',y') ) \leq a_{j}\;d(x,x')+b_{j}\;d(y,y').$ Thus,
\begin{align*}
d(\phi_{ij}(x,y), \phi_{ij}(x',y') ) &=
d((\phi_i(x,y), \phi_j(y, \phi_i(x,y))), \,(\phi_i(x',y'), \phi_j(y', \phi_i(x',y'))))\\
&=\max\left[d(\phi_i(x,y),\,\phi_i(x',y')) ,\; d(\phi_j(y, \phi_i(x,y)),\,\phi_j(y', \phi_i(x',y')))\right]\\
&\leq \max\left[a_{i}\;d(x,x')+b_{i}\;d(y,y')  ,\; a_{j}\;d(y,y')+b_{j}\;d(\phi_i(x,y),\phi_i(x',y'))\right]\\
&\leq \max\left[a_{i}\;d(x,x')+b_{i}\;d(y,y')  ,\; a_{j}\;d(y,y')+b_{j}\;\{a_{i}\;d(x,x')+b_{i}\;d(y,y')\}\right]\\
&\leq \max\left[a_{i}\;d(x,x')+b_{i}\;d(y,y')  ,\; b_{j}a_{i}\;d(x,x')+ \{a_{j}+b_{j}b_{i}\}\;d(y,y')\right]\\
&\leq  \lambda d((x,y),(x',y')),
\end{align*}
since $d((x,y),(x',y')) <  \max d(x,x'), \;d(y,y')$, where $\lambda =\max_{j} [a_{j}+b_{j} ]<1$. Consider the fractal operator $F_{\hat{\mathcal{S}}}: \mathcal{K}^{*}(X^{2}) \to \mathcal{K}^{*}(X^{2})$ associated to $\hat{\mathcal{S}}$ given by
$F_{\hat{\mathcal{S}}}( K)= \bigcup_{j=0,1}  F_{\hat{\phi}_{j}}( K).$

The operator $F_{\hat{\mathcal{S}}^{2}}$ is contractive because $\hat{\phi}_{i}\circ \hat{\phi}_{j}$ are contractions. We notice that $$\displaystyle F_{\hat{\mathcal{S}}}^2 ( K) = \bigcup_{i,j=0,1} F_{\hat{\phi}_{i} \circ\hat{\phi}_{j}}( K) = F_{\hat{\mathcal{S}}^2} ( K).$$ From Dumitru~\cite{MR2879091}, Theorem 1.1, or Kunze~\cite{MR3014680} we get that there  exists a unique compact set $B \in \mathcal{K}^{*}(X^{2})$ such that $F_{\hat{\mathcal{S}}}^2(B)= F_{\hat{\mathcal{S}}^2}(B)=B$ and $H_{n+1}=F_{\hat{\mathcal{S}}^2}(H_n)\to B$ for any $H_0 \in \mathcal{K}^{*}(X^{2})$.

Obviously, $F_{\hat{\mathcal{S}}}^2(B)=B$ implies~\footnote{Indeed, if $F_{\hat{\mathcal{S}}}^2(B)=B$ then $F_{\hat{\mathcal{S}}}(F_{\hat{\mathcal{S}}}^2(B)) =F_{\hat{\mathcal{S}}}(B)$, that is, $F_{\hat{\mathcal{S}}}^2(F_{\hat{\mathcal{S}}}(B)) =F_{\hat{\mathcal{S}}}(B)$. Thus $F_{\hat{\mathcal{S}}}(B)=B$.} that $F_{\hat{\mathcal{S}}}(B)=B$. The set $B$ is the unique set with this property. We call $A(\hat{\mathcal{S}})=B$ the attractor of $\hat{\mathcal{S}}$. From the fix point property we get that $\hat{H}_{n+1}=F_{\hat{\mathcal{S}}}(\hat{H}_n)\to A(\hat{\mathcal{S}})$ for any $\hat{H}_0 \in \mathcal{K}^{*}(X^{2})$.
\end{proof}

In the next example we provide some computational evidence for  the Proposition~\ref{doubleattractor}. In these approximation, we can see that $A(\hat{\mathcal{S}}) \subseteq A(\mathcal{S})^2$.

\begin{example}\label{proj does not} We consider the GIFS $\phi_{j}(x, y)= \frac{1}{3} x + \frac{(-1)^{j}}{4} y + \frac{j}{2}$, for $j=0,1$. A direct computation~\footnote{This computation and his consequence was pointed by one of the referees of this paper.} shows that $\phi_{0}([0,\frac{3}{4}], [0,\frac{3}{4}])=[0,\frac{7}{16}]$ and $\phi_{1}([0,\frac{3}{4}], [0,\frac{3}{4}])=[\frac{5}{16}, \frac{3}{4}]$ thus $F_{\mathcal{S}} ([0,\frac{3}{4}]\times [0,\frac{3}{4}])=[0,\frac{7}{16}]\cup [\frac{5}{16}, \frac{3}{4}]= [0,\frac{3}{4}]= A(\mathcal{S})$. Since the Hausdorff  dimension of $A(\mathcal{S})$ is one, and the Hausdorff  dimension of $\mathrm{proj}_x A(\hat{\mathcal{S}}) $ is less or equal to the dimension of $A(\hat{\mathcal{S}})$ that is strictly less than one it is impossible to have $\mathrm{proj}_x A(\hat{\mathcal{S}}) = A(\mathcal{S})$.

 Running the correspondent chaos game with 10000 iterations we get the picture of $\mathrm{proj}_x A(\hat{\mathcal{S}})$ and $A(\hat{\mathcal{S}})$ in the Figures \ref{attracgifs2} and \ref{attracgifs1}.
\end{example}

\begin{figure}[h!]
  \centering
  \includegraphics[width=8cm]{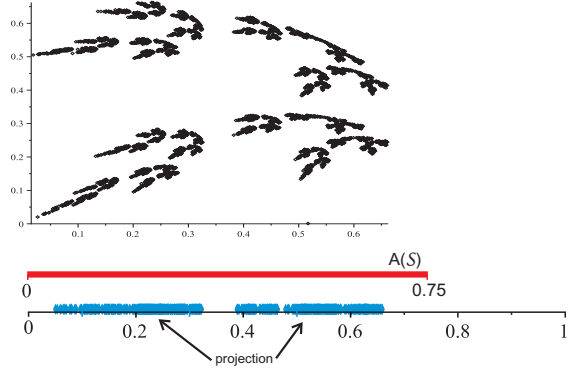}\\
  \caption{$\mathrm{proj}_x A(\hat{\mathcal{S}}) \subsetneq A(\mathcal{S})=[0,\frac{3}{4}]$, the attractor of $\mathcal{S}$. }\label{attracgifs2}
\end{figure}

\begin{figure}[h!]
\centering %voltar pra 6.5
  \includegraphics[width=4.5cm]{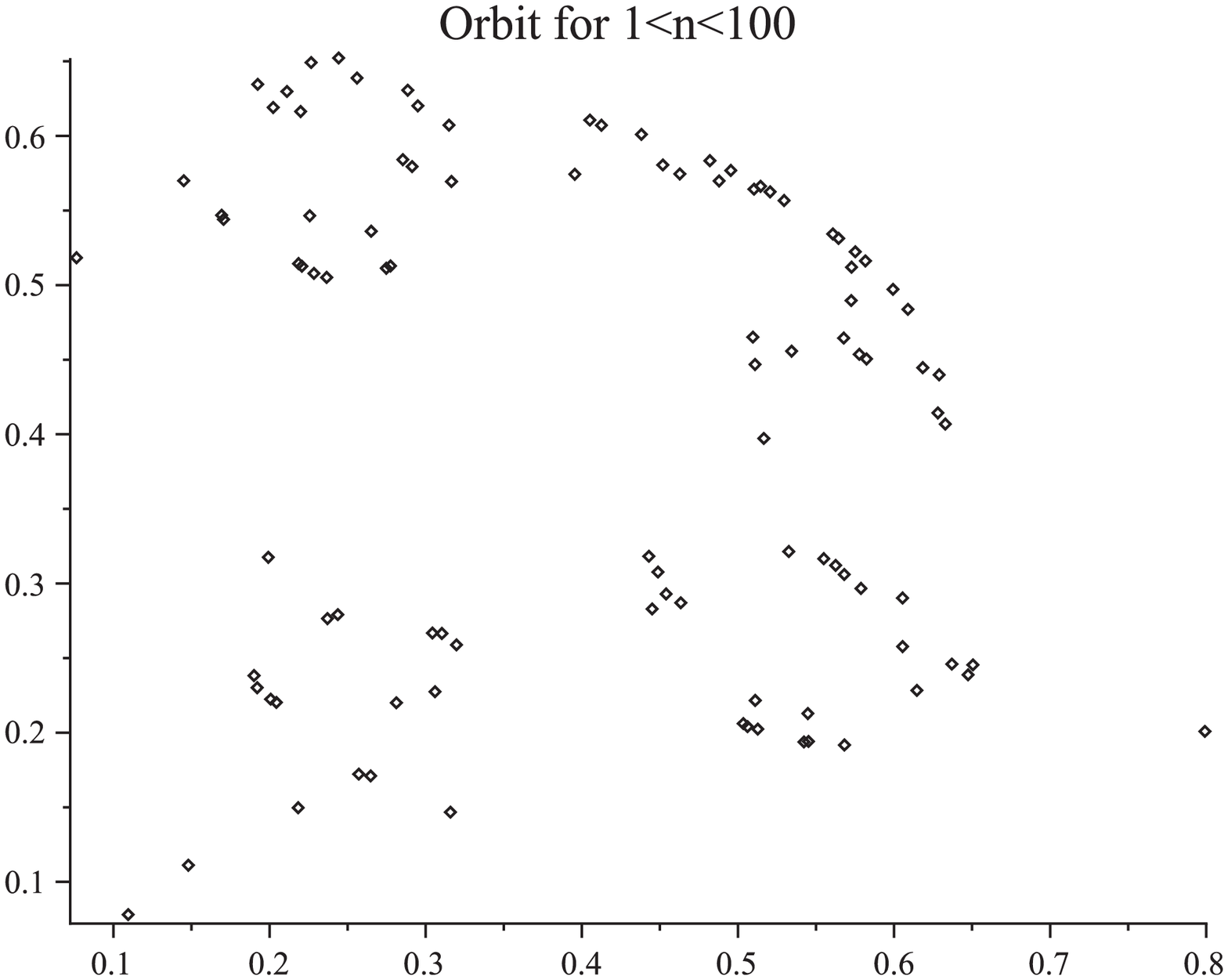} \hspace{.4cm}
  \includegraphics[width=4.5cm]{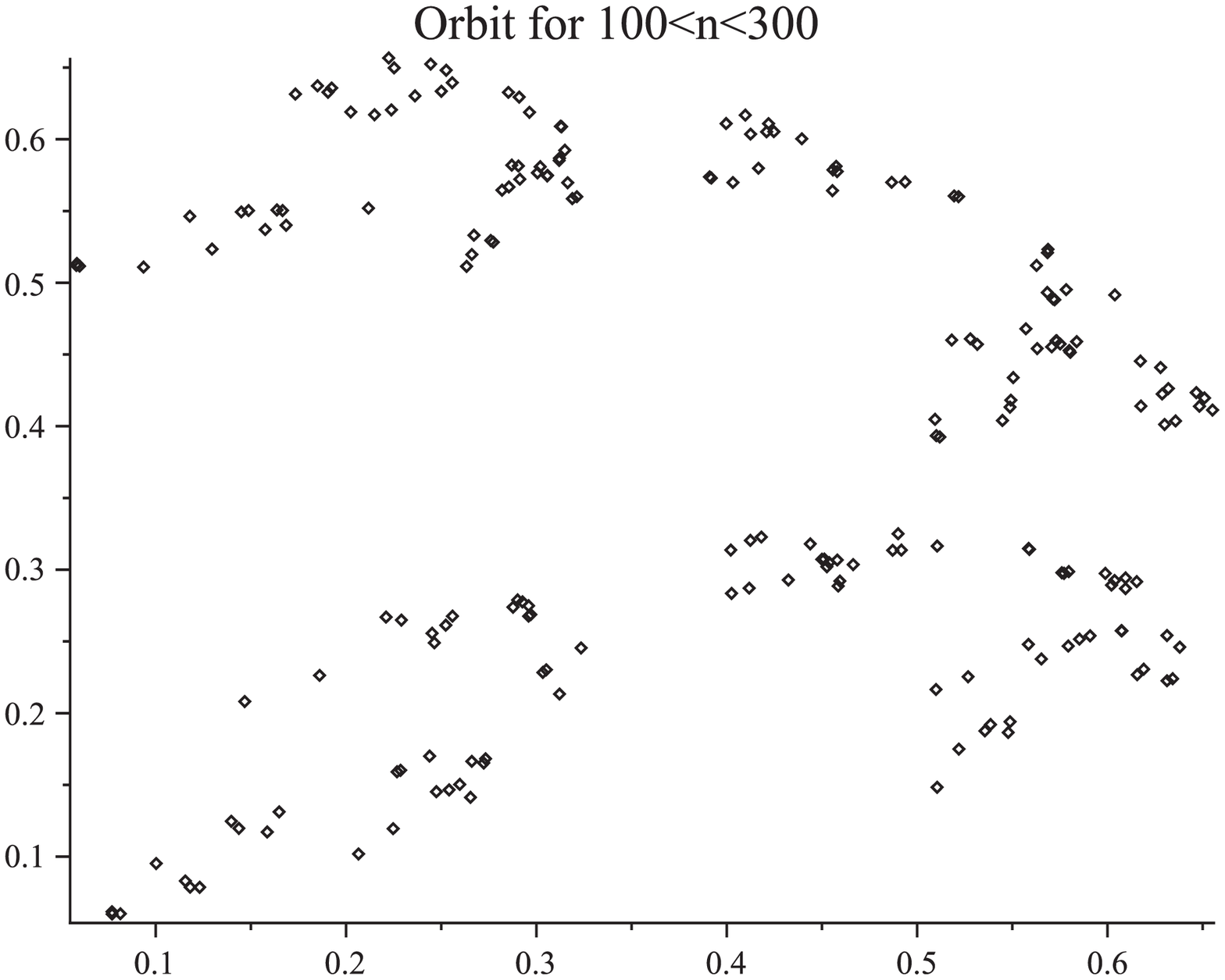}\\ \vspace{.4cm}
  \includegraphics[width=4.5cm]{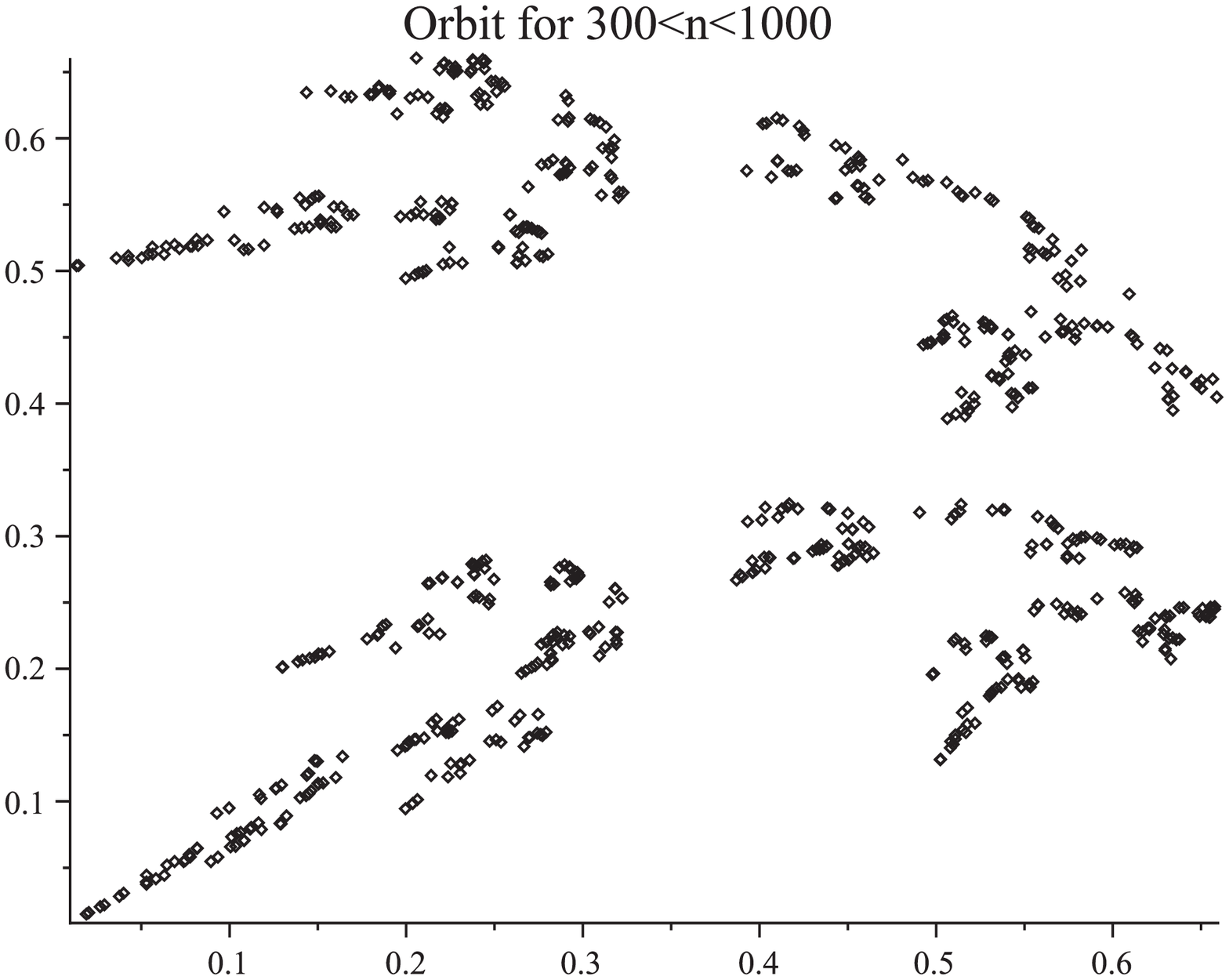}\hspace{.4cm}
  \includegraphics[width=4.5cm]{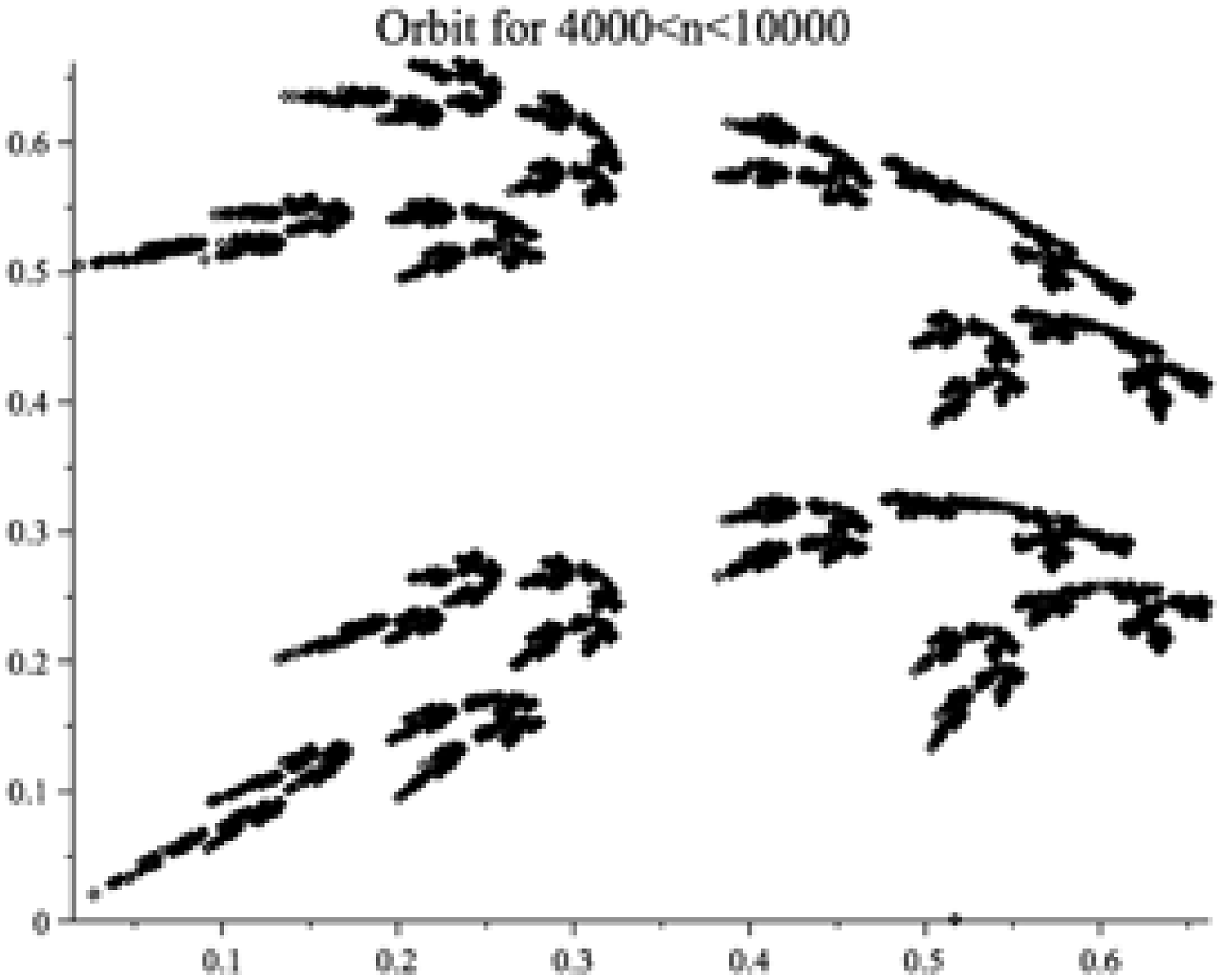}
  \caption{Building  $A(\hat{\mathcal{S}})$, the attractor of $\hat{\mathcal{S}}$, through a random orbit. }\label{attracgifs1}
\end{figure}

\begin{proposition}\label{doubleattractor} Let $A(\mathcal{S})$ be the attractor of $\mathcal{S}$, then $A(\mathcal{S})^2$ is  forward invariant with respect to  the IFS $\hat{\mathcal{S}}$. Moreover, $A(\hat{\mathcal{S}}) \subseteq A(\mathcal{S})^2$ in particular, ${\rm proj}_{x} A(\hat{\mathcal{S}}) \subseteq A(\mathcal{S})$.
\end{proposition}
\begin{proof}
Consider the fractal operator $F_{\hat{\mathcal{S}}}: \mathcal{K}^{*}(X^{2}) \to \mathcal{K}^{*}(X^{2})$ associated to $\hat{\mathcal{S}}$.
We notice that for $K= K_1 \times K_2$ we have
$\displaystyle F_{\hat{\mathcal{S}}}( K_1 \times  K_2)= \bigcup_{j=0,1}  F_{\hat{\phi}_{j}}( K_1, K_2).$

Since $A(\mathcal{S})$ is self-similar ($ F_{\mathcal{S}}(A(\mathcal{S}), A(\mathcal{S}))= A(\mathcal{S})$) we get
\begin{align*}
  F_{\hat{\mathcal{S}}}( A(\mathcal{S})^2) &= \bigcup_{j=0,1} F_{\hat{\phi}_{j}}( A(\mathcal{S})^2)= \bigcup_{j=0,1}  \hat{\phi}_{j}( A(\mathcal{S})^2)\\
  &= \bigcup_{j=0,1}  \{(y,  \phi_{j}(x,y))| x,y \in A(\mathcal{S})\}  \\
  &\subseteq \bigcup_{j=0,1}  ( A(\mathcal{S}),  \phi_{j}( A(\mathcal{S})^2))  \\
  &=(A(\mathcal{S}), F_{\mathcal{S}}( A(\mathcal{S})^2))= A(\mathcal{S})^2,
\end{align*}
that is, $A(\mathcal{S})^2$ is  forward invariant with respect to the IFS $\hat{\mathcal{S}}$.

Taking $H_0 = H_{1}=A(\mathcal{S}) \in \mathcal{K}^{*}(X)$ and $\hat{H}_0 = H_0 \times H_{1}=A(\mathcal{S})^2 \in \mathcal{K}^{*}(X^{2})$ the recursive sequence of compact subsets $\hat{H}_{n+1}= F_{\hat{\mathcal{S}}}( \hat{H}_{n}) \subseteq A(\mathcal{S})^2$ thus, $A(\hat{\mathcal{S}}) \subseteq A(\mathcal{S})^2$.
\end{proof}

\begin{remark}\label{eventually contract 2 probabilities} To study higher powers of a GIFSpdp~\footnote{We recall that the IFSpdp  $\hat{\mathcal{S}}^2=\left(X^{2},  (\hat{\phi}_{ij} )_{_{i,j\in\{0,1\}^2}}, (p_{ij})_{_{i,j\in\{0,1\}^2}}\right),$ is the 2nd-power extension of $\mathcal{S},$ where each $\hat{\phi}_{ij}: X^{2} \to X^{2}$ is defined by
$\hat{\phi}_{ij} (x,y) = (\hat{\phi}_i \circ \hat{\phi}_j )(x,y)= (\phi_j(x,y), \phi_i(y, \phi_j(x,y))),$ and
$p_{ij}(x,y)=p_{i}(x,y) p_{j}(\hat{\phi}_{i}(x,y)).$ The same definition holds for $\hat{\mathcal{S}}^m$,\; $m\geq 2$.} we need also consider the regularity of the weights in the IFSpdp $\hat{\mathcal{S}}=\left(X^{2},  (\hat{\phi}_i )_{i=0,1}, (\hat{p}_j)_{j=0,1} \right).$ In fact, if $p_i(x,y) \in Lip_{c_{i}, d_{i}}(X^2, [0,1])$ with $c_{i}+ d_{i}<1$ then,
$p_{ij}(x,y)$ satisfy $|p_{ij}(x,y)- p_{ij}(x',y')|\leq 2\; d((x,y), (x',y')$. In particular each $p_{ij}$ is Dini continuous with modulus of continuity $Q_{ij}(t) = 2\;t$.

To see this, we compute the distance
\begin{align*}
  |p_{ij}(x,y)- p_{ij}(x',y')| & = |p_{i}(x,y) p_{j}(\hat{\phi}_{i}(x,y))- p_{i}(x',y') p_{j}(\hat{\phi}_{i}(x',y'))| \\
   & \leq |p_{i}(x,y) p_{j}(\hat{\phi}_{i}(x,y))- p_{i}(x ,y ) p_{j}(\hat{\phi}_{i}(x',y'))|\\
   &+|p_{i}(x,y) p_{j}(\hat{\phi}_{i}(x',y'))- p_{i}(x',y') p_{j}(\hat{\phi}_{i}(x',y'))| \\
   & \leq |p_{i}(x,y)| \; |p_{j}(\hat{\phi}_{i}(x,y))-  p_{j}(\hat{\phi}_{i}(x',y'))|\\
   &+|p_{i}(x,y)- p_{i}(x',y')| \; |p_{j}(\hat{\phi}_{i}(x',y'))| \\
   & \leq  c_{j} d(y,y') + d_{j} d(\phi_{i}(x,y),\phi_{i}(x',y'))\\
   &+ c_{i} d(x,x') + d_{i}d(y,y')\\
   & \leq  c_{j} d(y,y') + d_{j} [a_{i} d(x,x') + b_{i}d(y,y')]\\
   &+ c_{i} d(x,x') + d_{i}d(y,y')\\
   &\leq [c_{i} + d_{j} a_{i}] d(x,x') + [c_{j} + d_{j}b_{i} + d_{i}]d(y,y')\\
   &= q_{ij} d(x,x') + r_{ij}d(y,y'),
\end{align*}
where $q_{ij}+ r_{ij}= [c_{i} + d_{j} a_{i}] + [c_{j} + d_{j}b_{i} + d_{i}] <2$.

\end{remark}

\subsection{Extended GIFS and the holonomic condition}
Given a GIFSpdp $\mathcal{S}=\left(X, (\phi_j)_{j=0,1}, (p_j)_{j=0,1}\right)$ we can to embed it, in to an IFSpdp   $\hat{\mathcal{S}}$ by setting $\hat{\mathcal{S}}=\left(X^{2},  (\hat{\phi}_j )_{j=0,1}, (\hat{p}_j)_{j=0,1} \right).$ As before $\hat{\phi}_j (x,y) = (y, \phi_j(x,y)),$ but we keep $\hat{p}_j = p_j: X^{2} \to [0,1]$. We consider the holonomic structure for an IFS introduced in Lopes and Oliveira~\cite{MR2461833}. Let $\hat{\sigma} : X^2 \times \Omega \to X^2 \times \Omega$ be the skill map

$$\hat{\sigma}((x,y),w)=\left(\hat{\phi}_{X_{0}(w)}(x,y), \sigma(w)\right)= \left(y, \phi_{X_{0}(w)}(x,y), \sigma(w)\right),$$
where $\Omega=\{0,1\}^{\mathbb{N}}$, $\sigma: \Omega \to\Omega$ is given by $\sigma(w_{0}, w_{1}, ...)=(w_{1}, w_{2}, ...)$ and $X_{k}: \Omega \to \{0,1\}$ is the projection on the coordinate $k$. A cylinder is the set $\overline{w_0 \cdots w_{n-1}}=\{ w \in \Omega \;|\; X_{0}(w)= w_0, ...,  X_{n-1}(w)=w_{n-1}\} \subset \Omega$. We recall the recurrence relation $x_{0} = Z_{0, w} (x_0, x_1)=x$, $x_{1}= Z_{1, w} (x_0, x_1)=y $ and  $x_{j} = Z_{j, w} (x_0, x_1)$ for $j\geq 2$ where $Z_{j+1, w} =  \phi_{w_{j-1}}(Z_{j-1 , w}, Z_{j, w}), \; j \geq 1$, so  $$\hat{\sigma}^{n}(x,y,w)=(Z_{n,w}(x,y), Z_{n+1,w}(x,y),\sigma^{n}(w))= (x_{n} , x_{n+1} ,\sigma^{n}(w)).$$
The ergodic averages evaluated on the orbits of $\hat{\sigma}^{n}$ allow us to define certain measures in $X^2 \times \Omega$ that captures the behavior of the IFSpdp
\begin{align*}
  \int_{X^2 \times \Omega} h(x,y,w) d\eta_{_{N,a, (x_0, x_1)}}(x,y,w) &=\frac{1}{N} \sum_{i=0}^{N-1}h(\hat{\sigma}^{n}(x,y,a)) \\
 &=\frac{1}{N} \sum_{i=0}^{N-1}h(x_{i}(a) , x_{i+1}(a)  , \sigma^{i}(a)),
\end{align*}
for all $ h \in C( X^2 \times \Omega , \mathbb{R})$.

Any weak limit of $(\mu_{_{N, a, (x_0, x_1)}})_{_{N\geq 0}}$ is a measure  $\mu_{a, (x_0, x_1)}$, that satisfy $ \int_{X^2 \times \Omega} g \circ \hat{\sigma} \; d\mu_{a, (x_0, x_1)}=  \int_{X^2 \times \Omega} g \; d\mu_{a, (x_0, x_1)},\quad \forall g \in C(X^2, \mathbb{R}).$  This measures are called holonomic  in Lopes and Oliveira \cite{MR2461833}. Since $Prob(X^2 \times \Omega)$ is compact we know that the set of those probabilities is obviously  not empty.

\begin{definition}The set of holonomic probabilities with respect to the  extended GIFS $\hat{\mathcal{S}}=\left(X^{2},  (\hat{\phi}_i (x,y))_{i=0,1} \right)$ is the set of probabilities $\eta$ in $X^2 \times \Omega$ such that
$ \int_{X^2 \times \Omega} g \circ \hat{\sigma} \; d\eta=  \int_{X^2 \times \Omega} g \; d\eta,
\quad \forall g \in C(X^2, \mathbb{R}).$ In other words $$ \displaystyle \int_{X^2 \times \Omega} g(\hat{\phi}_{X_{0}(w)}(x,y))  \; d\eta=  \int_{X^2 \times \Omega} g(x,y) \; d\eta,
\quad \forall g\in C(X^2, \mathbb{R}).$$
\end{definition}

\begin{lemma}\label{holondisint}
Any holonomic measure $\eta$  induces  a \textbf{measurable} GIFSpdp
$\hat{\mathcal{S}}=\left(X^{2},  (\hat{\phi}_i)_{i=0,1}, (p_i)_{i=0,1} \right),$
where $p_i(x,y)=J_{(x,y)}(\overline{i}), i=0,1$ are measurable weights and  $d\eta(x,y,w)=dJ_{(x,y)}(w) \; d\alpha(x,y)$ is the disintegration of $\eta$. If $\eta$ is such that $p_i(x,y)=J_{(x,y)}(\overline{i}), i=0,1$ are continuous then $\hat{\mathcal{S}}$ is actually a (continuous, finite, ...) GIFSpdp.
\end{lemma}
\begin{proof}
By disintegration (see Dellacherie \cite{MR521810}) we have, for any holonomic measure $\eta$, a decomposition $d\eta(x,y,w)=dJ_{(x,y)}(w) \; d\alpha(x,y)$,
$$ \int_{X^2 \times \Omega} h(x,y,w)  \; d\eta(x,y,w)= \int_{X^2 } \int_{\Omega} h(x,y,w) \; dJ_{(x,y)}(w) \; d\alpha(x,y),$$
$\quad \forall h\in C(X^2 \times \Omega, \mathbb{R}),$ where the family of probability kernels $J_{(x,y)}$ is unique $\alpha$ almost everywhere.
A particular case is  $h(x,y,w)= g(\hat{\phi}_{X_{0}(w)}(x,y))$ for $g \in C(X^2, \mathbb{R})$. Then, the holonomic condition turn in to
$$ \int_{X^2 } \int_{\Omega} g(\hat{\phi}_{X_{0}(w)}(x,y)) \; dJ_{(x,y)}(w) \; d\alpha(x,y)= \int_{X^2 } \int_{\Omega} g(x,y) \; dJ_{(x,y)}(w) \; d\alpha(x,y)$$
$$ \int_{X^2 } \sum_{i=0,1} J_{(x,y)}(\overline{i}) g(\hat{\phi}_{i}(x,y))  \; d\alpha(x,y)= \int_{X^2 }  g(x,y)  \; d\alpha(x,y),$$
$\forall g\in C(X^2, \mathbb{R}).$
In other words, the marginal $\alpha(x,y)$ is an invariant measure for the dual of the operator $B_{\hat{\mathcal{S}}}(g) (x,y)= \sum_{i=0,1} J_{(x,y)}(\overline{i}) g(\hat{\phi}_{i}(x,y)):$
$$ \int_{X^2 } B_{\hat{\mathcal{S}}}(g) (x,y) \; d\alpha(x,y)= \int_{X^2 }  g(x,y)  \; d\alpha(x,y),$$
$\forall g\in C(X^2, \mathbb{R}).$
\end{proof}

The Markov operator associated to the extended GIFSpdp  arises naturally from Lemma~\ref{holondisint} and makes us to consider a new class of Markov operators for a GIFSpdp.

\begin{definition} The \textbf{extended operators} associated to the GIFSpdp $\mathcal{S}$ are the usual operators to the IFSpdp $\hat{\mathcal{S}}$:
\begin{enumerate}
  \item[1- ] $B_{\hat{\mathcal{S}}}: C(X^2 , \mathbb{R}) \to C(X^2, \mathbb{R})$  by
$$ B_{\hat{\mathcal{S}}}(g) (x,y)= \sum_{j=0,1} p_{j}(x,y) g(\hat{\phi}_{j}(x,y)),\; \forall (x,y) \in X^2.$$
  \item[2- ] $\mathcal{L}_{\hat{\mathcal{S}}} : Prob(X^2) \to Prob(X^2)$ by
$$\int_{X^2} g(x,y) d\mathcal{L}_{\hat{\mathcal{S}}} (\alpha)(x,y) = \int_{X^{2}} B_{\hat{\mathcal{S}}}(g) (x,y) d\alpha(x,y),$$
for any $\alpha \in Prob(X^2)$ and $g: X^2 \to \mathbb{R}$.
\end{enumerate}
\end{definition}

\begin{definition} The $\eta$-operators  associated to the holonomic measure $d\eta(x,y,w)=dJ_{(x,y)}(w) \; d\alpha(x,y)$ are the extended operators $B_{\hat{\mathcal{S}}}$ and $\mathcal{L}_{\hat{\mathcal{S}}}$ associated to the GIFSpdp $\hat{\mathcal{S}}=\left(X^{2},  (\hat{\phi}_i )_{i=0,1}, (p_i)_{i=0,1} \right),$
where $p_i(x,y)= J_{(x,y)}(\overline{i})=J_{(x,y)}(\{w|w_{0}=i\}).$ By Lemma~\ref{holondisint}, $\alpha$ is a fixed point,
$\mathcal{L}_{\hat{\mathcal{S}}}(\alpha)= \alpha.$
\end{definition}

Generally speaking, the weights $p_i(x,y)= J_{(x,y)}(\overline{i})=J_{(x,y)}(\{w|w_{0}=i\})$ are not continuous. However it can happen if we start with the Markov process with place dependent probabilities, associated to a GIFSpdp (a Feller process).

\begin{definition} \label{HolLiftDef} Given a fixed point \footnote{There exists fixed points  because $\mathcal{L}_{\hat{\mathcal{S}}}$ is a continuous operator and $Prob(X^2)$ is compact and convex (Schauder fixed point theorem).} $\alpha$ of the extended Markov operator  ($\mathcal{L}_{\hat{\mathcal{S}}} (\alpha)=\alpha$) associated to a GIFSpdp $\mathcal{S}=\left(X, (\phi_j)_{j=0,1}, (p_j)_{j=0,1}\right),$ one can introduce a measure  in $\Omega$ indexed by $(x_1=y,x_0=x)$ by considering his orbit $x_n=Z_{n,w}(x,y)$ as a Markov process with place dependent probabilities. We assign probabilities to each cylinder $\overline{w_0 \cdots w_{n-1}} \subset \Omega$ by
$$P_{(x,y)}(\overline{w_0 \cdots w_{n-1}})=p_{w_0}(x_0 ,x_1 ) \cdots p_{w_{n-1}}(x_{n-1}, x_{n}).$$
The product measure $d\eta(x,y,w)=dP_{(x,y)}(w)d\alpha(x,y),$ is a holonomic measure called the \textbf{holonomic lifting}~(see Lopes and Oliveira \cite{MR2461833}) of  the measure $\alpha$.
Indeed,
\begin{align*}
\int_{X^2 \times \Omega} g(\hat{\phi}_{X_{0}(w)}(x,y))  \; d\eta&= \int_{X^2 }\int_{ \Omega} g(\hat{\phi}_{X_{0}(w)}(x,y))  \; dP_{(x,y)}(w)d\alpha(x,y) \\
&= \int_{X^2 }B_{\hat{\mathcal{S}}}(g) (x,y) d\alpha(x,y)= \int_{X^2 }g(x,y)d\mathcal{L}_{\hat{\mathcal{S}}}\alpha(x,y)\\
&=  \int_{X^2} g(x,y) \; d\alpha(x,y)=  \int_{X^2 \times \Omega} g(x,y) \; d\eta.
\end{align*}
\end{definition}

We recall that the IFSpdp  $\hat{\mathcal{S}}^2=\left(X^{2},  (\hat{\phi}_{ij} )_{_{i,j\in\{0,1\}^2}}, (p_{ij})_{_{i,j\in\{0,1\}^2}}\right),$ is the 2nd-power extension of $\mathcal{S},$ where each $\hat{\phi}_{ij}: X^{2} \to X^{2}$ is defined by
$\hat{\phi}_{ij} (x,y) = (\hat{\phi}_i \circ \hat{\phi}_j )(x,y)= (\phi_j(x,y), \phi_i(y, \phi_j(x,y))),$ and
$p_{ij}(x,y)=p_{i}(x,y) p_{j}(\hat{\phi}_{i}(x,y)).$

\begin{lemma}\label{Power 2 IFSpdp fixed} Let  $\hat{\mathcal{S}}$ be an IFSpdp then, $\mathcal{L}_{\hat{\mathcal{S}}}$ has a unique fixed point $\alpha$(ergodic with respect to the Markov process generated by the IFSpdp). In particular, $\supp(\alpha)=A(\hat{\mathcal{S}})$. \end{lemma}
\begin{proof} The proof uses the fact that $\hat{\mathcal{S}}$ is eventually contractive. More precisely, we show in  Lemma~\ref{eventually contract 2} that $\hat{\mathcal{S}}^{2}$ is contractive and has an attractor $A(\hat{\mathcal{S}})$. From Barnsley~\cite{MR971099}, Theorem 2.1~\footnote{In \cite{MR971099}, they assume that the IFSpdp satisfy a Dini-type condition that is weaker than E2 for 2nd power extension. More precisely, the weights are average-contractive: $\sum p_{ij} \ln Lip(\hat{\phi}_{ij}) <0$.} or Kunze~\cite{MR3014680}, Theorem 2.60 and Theorem 2.63, there is a unique fixed point  $\alpha \in {\rm Prob}(X^2)$, for $\mathcal{L}_{\hat{\mathcal{S}}^{2}}$ and $\supp(\alpha)$ is the attractor of $\hat{\mathcal{S}}^{2}$ (that is equal to $A(\hat{\mathcal{S}})$ by Lemma~\ref{eventually contract 2}).

Let $\bar{\alpha}$ be any fixed point  of $\mathcal{L}_{\hat{\mathcal{S}}}$ that is, $\mathcal{L}_{\hat{\mathcal{S}}} \bar{\alpha}= \bar{\alpha}$.
One can show that
$$B_{\hat{\mathcal{S}}}^{2}(g) (x,y)= \sum_{i,j =0,1} p_{i}(x,y) p_{j}(\hat{\phi}_{i}(x,y)) g(\hat{\phi}_{i}(\hat{\phi}_{j}(x,y)))=B_{\hat{\mathcal{S}}^{2}}(g) (x,y),$$
so,
\begin{align*}
  \int_{X^2} g(x,y) \; d\bar{\alpha}(x,y)&= \int_{X^2} g(x,y) d\mathcal{L}_{\hat{\mathcal{S}}} \bar{\alpha}(x,y) = \int_{X^2} g(x,y) d\mathcal{L}_{\hat{\mathcal{S}}}^{2} \bar{\alpha}(x,y) \\
  &= \int_{X^2} B_{\hat{\mathcal{S}}}^{2}(g) (x,y) d\bar{\alpha}(x,y) = \int_{X^2} B_{\hat{\mathcal{S}}^{2}}(g) (x,y) d\bar{\alpha}(x,y) \\
  &= \int_{X^2} g(x,y) d\mathcal{L}_{\hat{\mathcal{S}}^{2}} \bar{\alpha}(x,y),
\end{align*}
that is, $\mathcal{L}_{\hat{\mathcal{S}}^{2}} \bar{\alpha}= \bar{\alpha}$, thus $\bar{\alpha}=\alpha$.
\end{proof}

\begin{definition} The unique fixed point of $\mathcal{L}_{\hat{\mathcal{S}}} $ in $Prob(X^2)$ is the \textbf{Extended Hutchinson measure} for $\mathcal{S}$.
\end{definition}

We can also to apply $B_{\hat{\mathcal{S}}}$ in functions $g(x,y)= f(\Pi_y (x,y))$, where $\Pi_y (x,y)=y$ is the projection on the second coordinate and $f\in C(X, \mathbb{R})$, obtaining
$\displaystyle B_{\hat{\mathcal{S}}}(g) (x,y)= \sum_{j=0,1} p_{j}(x,y) g(\hat{\phi}_{j}(x,y))= \sum_{j=0,1} p_{j}(x,y) f( \phi_{j}(x,y)) $ $= B_{\mathcal{S}}(f) (x,y)$
that is,  $B_{\hat{\mathcal{S}}}(f\circ\Pi_y) (x,y)=B_{\mathcal{S}}(f) (x,y),$  for any $f\in C(X, \mathbb{R})$.

\begin{lemma}\label{IdentcMargin} If $\mathcal{L}_{\hat{\mathcal{S}}} \alpha= \alpha$ then $(\Pi_x)_* \alpha = (\Pi_y)_* \alpha$~\footnote{The subscript ``*" means the push-forward map, $T_* \alpha$. For $T:X^2 \to X$, the push-forward map is defined by $\int_{X} f d T_* \alpha=\int_{X^2} f\circ T d\alpha$.}, that is, the marginals of the extended Hutchinson measure $\alpha$ are the same. The measure $\mu_{\alpha}=(\Pi_x)_* \alpha = (\Pi_y)_* \alpha$ is called \textbf{the projected Hutchinson measure}.
\end{lemma}
\begin{proof}
If we take $g(x,y)= f(\Pi_x (x,y))$, where $\Pi_x (x,y)=x$ is the projection on the first coordinate and $f \in C(X, \mathbb{R})$ we have
$$B_{\hat{\mathcal{S}}}(g) (x,y)= \sum_{j=0,1} p_{j}(x,y) g(\hat{\phi}_{j}(x,y))=$$ $$=\sum_{j=0,1} p_{j}(x,y) (f\circ \Pi_x)(y, \phi_{j}(x,y))= f(y)\sum_{j=0,1} p_{j}(x,y)= f(y)$$
that is  $B_{\hat{\mathcal{S}}}(f\circ\Pi_x) (x,y)=f(y),$  for any $f\in C(X, \mathbb{R})$. Since $\mathcal{L}_{\hat{\mathcal{S}}}(\alpha)= \alpha$ we have
$$\int_{X^2 } f(y) d\alpha(x,y)= \int_{X^2 }  B_{\hat{\mathcal{S}}}(f\circ\Pi_x) (x,y) d\alpha(x,y)= $$ $$=\int_{X^2 }(f\circ\Pi_x) (x,y) d\mathcal{L}_{\hat{\mathcal{S}}}\alpha(x,y) = \int_{X^2 }f (x) d\alpha(x,y),$$
for any $f\in C(X, \mathbb{R})$.
\end{proof}

\subsection{The ergodic theorem for GIFS}
We address the problem of  using ergodic averages of the extended GIFS to estimate the integrals $\int_{X^2} g(x,y) d\alpha(x,y)$ where  $\alpha$ is the extended Hutchinson measure.

\begin{remark} \textbf{  The results will be stated for a general GIFS of degree $m$ with $n$ maps but the proofs will be made just for $m=2$ avoiding the extra indexes, because there is no difference at all, in the reasoning. In the rest of this section, the sequence  $x_i(w) = Z_{i, w} (x_0, x_1)$ for $i=0, 1, ...$ is obtained from  $(x_0, x_1)$ by the iteration by $\mathcal{S}$.}
\end{remark}

Our main tool is the Ergodic Lemma.
\begin{lemma} \label{extergtheorem} Let  $\alpha$ be the extended Hutchinson measure. For each $(x_0, x_1, ..., x_{m-1})\in X^m$ there exists a measurable set $\Omega_{(x_0,  x_1, ..., x_{m-1})}\subseteq \Omega=\{0, ..., n-1\}^{\mathbb{N}}$ such that $P_{(x_0, x_1, ..., x_{m-1})}(\Omega_{(x_0, x_1, ..., x_{m-1})})=1$ and  for any $w \in \Omega_{(x_0, x_1, ..., x_{m-1})}$
$$\frac{1}{N} \left(g(x_{ 0},... ,x_{m-1})+ ... +g(x_{mN-m},... , x_{mN-1})\right) \\ \to \int_{X^2} g(x) d\alpha(x),$$
$ \forall \; g \in C( X^m , \mathbb{R}).$
\end{lemma}
\begin{proof} The proof is based on the Elton's ergodic theorem for IFSpdp (see  Elton~\cite{MR922361}).  In order to do that we need to use probabilities for $\hat{\mathcal{S}}^{2}$ because it is contractive.

Let $P_{(x,y)}^{2}$ be the measure in $\Omega^{2}=(\{0,1\}\times\{0,1\})^{\mathbb{N}}$ indexed by $(x=x_0, y=x_1)$~\footnote{ At this point, if the GIFS has degree $m$ we should to consider $P_{(x_0,...,x_{m-1})}^{m}$, the measure in $\Omega^{m}=(\{0,1\}\times \cdots \times \{0,1\})^{\mathbb{N}}$.}.  We built this measure  by considering his orbit $x_n(w)=Z_{n,w}(x,y)$  by the GIFS as a Markov process of higher order with place dependent probabilities
$(x_1=y,x_0=x),$ $ (x_2,x_3)=\hat{\phi}_{w_{1}w_{0}}(x_0 ,x_1),$   $(x_4,x_5)=\hat{\phi}_{w_{3}w_{2}}(x_2 ,x_3),  \cdots.$

We assign probabilities on each cylinder $\overline{(w_0,w_1) \cdots (w_{2n-2},w_{2n-1})} \subset \Omega^{2}$ (that generates the Borel sigma algebra of $\Omega^{2}$) as follows
$$P_{(x,y)}^{2}(\overline{(w_0,w_1) \cdots (w_{2n-2},w_{2n-1})})= p_{w_{0}w_{1}}(x_0 ,x_1 ) \cdots p_{w_{n-1}w_{n}}(x_{n-1}, x_{n}).$$

The  map $\Gamma :  \Omega^{2} \to \Omega$ given by
$\Gamma((w_0,w_1) \cdots (w_{2n-2},w_{2n-1}),...)= (w_0 \cdots w_{2n-1},...),$
is obviously a homeomorphism. Applying this map on cylinders with even lengths we get
$$\Gamma(\overline{(w_0,w_1) \cdots (w_{2n-2},w_{2n-1})})= \overline{w_0 \cdots w_{2n-1}},$$
and for cylinders with odd length we get
$$\Gamma^{-1}(\overline{w_0 \cdots w_{2n}})=  \overline{(w_0,w_1) \cdots (w_{2n},0)} \cup \overline{(w_0,w_1) \cdots (w_{2n},1)},$$
so $\Gamma$ is bi-measurable because the respective Borel  sigma algebras are generated by the respective pre images. For each cylinder  $\overline{w_0 \cdots w_{2n-1}} \in \Omega$ we have
$$ P_{(x,y)}^{2}(\overline{(w_0,w_1) \cdots (w_{2n-2},w_{2n-2})})= P_{(x,y)}(\overline{w_0 \cdots w_{2n-1}}),$$
in particular $\Gamma$ preserves measure.

From Lemma~\ref{Power 2 IFSpdp fixed} we get that $\alpha$ is ergodic with respect to the $\hat{\mathcal{S}}^{2}$. The Elton's ergodic theorem actually requires an average contraction hypothesis $$\displaystyle\prod_{ij}d(\hat{\phi}_{ij}(x,y), \hat{\phi}_{ij}(x',y') )^{p_{ij}(x,y)} < r \; d((x,y), (x',y') ),$$ for some $ r < 1$, which is a consequence of E1 and Lemma~\ref{eventually contract 2}, and that each $p_{ij}(x,y)$ to be positive and Dini continuous which is a consequence of E2 and Remark~\ref{eventually contract 2 probabilities}. Thus, we can use the Elton's ergodic theorem for $\hat{\mathcal{S}}^{2}$, that is, there exists a set  $\Omega_{(x_0, x_1)}^{2}\subseteq \Omega^{2}$ with probability one, such that for any $w \in \Omega_{(x_0, x_1)}^{2}$,
$$\frac{1}{N} \left( g(x_{ 0}(w), x_{ 1}(w))+ ... +g(x_{2N-2}(w), x_{2N-1}(w))\right) \to \int_{X^2} g(x,y) d\alpha,$$
$ \forall \; g \in C( X^2 , \mathbb{R}).$ We get our result if we take
$\Omega_{(x_0, x_1)}=\Gamma(\Omega_{(x_0, x_1)}^{2})\subseteq \Omega,$
because, $$P_{(x_0, x_1)}(\Omega_{(x_0, x_1)})= P_{(x_0, x_1)}(\Gamma(\Omega_{(x_0, x_1)}^{2}))=  P_{(x_0, x_1)}^{2}(\Omega_{(x_0, x_1)}^{2})=1.$$
\end{proof}

The next theorem generalizes the Elton's Ergodic Theorem for GIFS, providing a basis for applications and further studies in  Chaos Games~\cite{MR971099, MR977274, MR1045252, MR2254477} for GIFS. This theorem generalizes the analogous result for IFS, Barnsley~\cite{MR2254477}, Pg. 323, Theorem 4.5.

\begin{theorem}\label{EltErgTheoremGIFS} (Elton's Ergodic Theorem for GIFS) Let  $\alpha$ be the extended Hutchinson measure for $\mathcal{S}$. For each $(x_0, x_1, ..., x_{m-1})\in X^m$ there exists a measurable set $\Omega_{(x_0,  x_1, ..., x_{m-1})}\subseteq \Omega=\{0, ..., n-1\}^{\mathbb{N}}$ such that $P_{(x_0, x_1, ..., x_{m-1})}(\Omega_{(x_0, x_1, ..., x_{m-1})})=1$ and  for any $w \in \Omega_{(x_0, x_1, ..., x_{m-1})}$
$$\frac{1}{K} \left( f(x_{K-1}(w))+ ... +f(x_{1}(w)) +f( x_{0}(w))\right) \to \int_{X^2} f d\mu_{\alpha},$$
$ \forall \; f \in C( X , \mathbb{R})$, where $\mu_{\alpha}=(\Pi_{x_{0}})_* \alpha = (\Pi_{x_{1}})_* \alpha= \cdots = (\Pi_{x_{m-1}})_* \alpha$  is the projected Hutchinson measure.\\
\end{theorem}
\begin{proof} Again, we use $m=2$ to simplify the writing.
Consider the projections $\Pi_x (x,y)=x$ and $\Pi_y (x,y)=y$. If we apply the Lemma~\ref{extergtheorem} to $g(x,y)= f(\Pi_x (x,y))$ and  $g(x,y)= f(\Pi_y (x,y))$ we get
$$\frac{1}{N} \left( (f\circ\Pi_y)(x_{2N-2}, x_{2N-1})+ ... +(f\circ\Pi_x)(x_{0}, x_{1})\right) \to \int_{X^2} (f\circ\Pi_y)(x,y) d\alpha,$$
$$\frac{1}{N} \left( f(x_{2N-1})+ ... +f(x_{3}) +f(x_{1})\right) \to \int_{X^2} f(y) d\alpha,$$
and
$$\frac{1}{N} \left( (f\circ\Pi_x)(x_{2N-2}, x_{2N-1})+ ... +(f\circ\Pi_y)(x_{0}, x_{1})\right) \to \int_{X^2} (f\circ\Pi_y)(x,y) d\alpha,$$
$$\frac{1}{N} \left( f(x_{2N-2})+ ...+f( x_{2}) +f( x_{0})\right) \to \int_{X^2} f(x) d\alpha.$$
Adding this two limits we get
$$\frac{2}{2N} \left( f(x_{2N-1})+ f(x_{2N-2})+ ...+f( x_{1}) +f( x_{0})\right) \to \int_{X^2}  f(x)+f(y) d\alpha,$$
thus
$$\frac{1}{K} \left( f(x_{K})+ ... +f(x_{1}) +f( x_{0})\right) \to \frac{1}{2}\int_{X^2} f(x)+f(y) d\alpha(x,y),$$
if $K=2N$. On the other hand, if $K=2N+1$ we use the fact that $\frac{\|f\|_{\infty}}{K}$ goes to zero uniformly, what guarantees the result for any $K$. Since $(\Pi_x)_* \alpha = \mu_{\alpha} = (\Pi_y)_* \alpha$, the limit above will be
$$\frac{1}{K} \left( f(x_{K-1}(w))+ ... +f(x_{1}(w)) +f( x_{0}(w))\right) \to \int_{X} f(x) d\mu_{\alpha}(x).$$
\end{proof}

\begin{remark}
A natural question is if  $\alpha=\mu \times \mu$. The answer is not in general, but it can hapens in some cases. One can easily test this possibility using functions $g(x,y)=a(x)b(y)$. Indeed, even for a standard GIFSpdp in $X=[0, 1]$, $\phi_0(x,y)= \frac{1}{4}x + \frac{1}{4}y \text{ and } \phi_1(x,y)= \frac{1}{4}x + \frac{1}{4}y + \frac{1}{2}$  and $p_j (x,y)=\frac{1}{2}, \; j=0,1$ we have, for $a(x)=x$ and a generic $b$, an impossible equation appears, if we suppose $\alpha=\mu \times \mu$.  In the Example~\ref{GIFSattrac2xmod1} of the Section~\ref{An example from Thermodynamic Formalism} we found $\alpha=\mu_{\mathcal{S}} \times \mu_{\mathcal{S}}$.
\end{remark}

\section{Applications and Examples}\label{applications}
In this section we use the ergodic theorem to get several consequences and applications to related fields.

\subsection{Chaos Game: a random iteration algorithm for GIFS} In the 80's, M. F. Barnsley \cite{MR977274, MR2254477} has introduced the idea of Chaos Game or Random Iteration Algorithm as a tool for drawing fractals appearing as attractors of IFS. However, in the last few years, the term Chaos Game means every iteration (sets, multifunction, point) using a random choice of maps. The algorithm is described as follows (see Barnsley~\cite{MR977274}, Chapter III, or Kunze~\cite{MR3014680}, Chapters 2 and 6 for more details):\\

\emph{Let $\mathcal{R}$ be an attractive IFS  on $X$, that  is a family of functions $\phi_j: X \to X$, and we introduce weight functions (probabilities) $p_j: X \to [0,1]$ such that  $p_0(x) +  p_{1}(x) =1$ producing an IFSpdp, denoted $\mathcal{R}=(X, (\phi_j)_{j=0,1}, (p_j)_{j=0,1})$.  If $A$ is the attractor of $\mathcal{S}$ then, for each $x_0 \in X$
$$A=\lim_{k\to \infty } \overline{\{x_n (a)\}_{n\geq k}},$$
for almost every $a\in \Omega=\{0,1\}^{\mathbb{N}}$ with respect to $P_{x_0}$~\footnote{The limit is taken with respect to the Hausdorff metric.}. Where the orbit $x_n (a)$ is obtained from $x_0$ taking $x_{n+1}=\phi_{a_{n}}(x_{n})$ and $a_{n}$ is chosen with probability ${\rm Prob}(a_n=j)=p_{j}(x_{n}).$}\\

We will show that the chaos game, in the sense of approximate the picture of the attractor by the closure of a random orbit, works for a extended GIFS. However, as we can see from Example~\ref{proj does not}, that satisfies the hypothesis of Theorem~\ref{chaosgametheorem}, we are not able to draw the attractor of GIFS by projecting the attractor of his extension.

\begin{theorem}\label{chaosgametheorem}(Chaos game for extended GIFS) Let $\mathcal{S}$ be a GIFS satisfying E1 and E2 hypothesis. If $A(\hat{\mathcal{S}})$ is the attractor of his extension then, for any fixed $x_0, x_1,... x_{m-1} \in X$ we have
$$A(\hat{\mathcal{S}})=\lim_{k\to \infty } \overline{\{(x_n (a),..., x_{n+m-1} (a)) \}_{n\geq k}},$$
for almost all $a \in \Omega$ with respect to the probability $P_{x_0,x_1,... x_{m-1}}$, where $\{x_n (a)\}_{n}$ is the orbit of $(x_0, x_1,... x_{m-1})$. In particular $\displaystyle\mathrm{proj}_{x} A(\hat{\mathcal{S}})=\lim_{k\to \infty } \overline{\{x_n (a)\}_{n\geq k}}$.
\end{theorem}
\begin{proof}
We assume $m=2$. To prove that $A(\hat{\mathcal{S}})=\lim_{k\to \infty } \overline{\{(x_n (a), x_{n+1} (a)) \}_{n\geq k}},$ we take $z \in A(\hat{\mathcal{S}})=\supp \alpha$ then there is a neighborhood $V$ of $z$ such that $\alpha(V)>0$. By Lemma~\ref{extergtheorem} there exists a measurable set $\Omega_{(x_0, x_1)}\subseteq \Omega$ such that $P_{(x_0, x_1)}(\Omega_{(x_0, x_1)})=1$ and for any $a \in \Omega_{(x_0, x_1)}$,
$\frac{1}{N} \sharp\left\{ _{ 0\leq j\leq N-1} \; | \; (x_{j}(a),x_{j+1}(a))   \in V \right\} \to \alpha(V).$ Reducing the size of $V$ we get that $z \in \lim_{k\to \infty } \overline{\{(x_{j}(a),x_{j+1}(a))\}_{n\geq k}}$ thus $A(\hat{\mathcal{S}}) \subseteq \lim_{k\to \infty } \overline{\{(x_{j}(a),x_{j+1}(a))\}_{n\geq k}}$. On the other hand, $A(\hat{\mathcal{S}})$ is forward invariant. If we have  $(x_0, x_1) \in A(\hat{\mathcal{S}})$ then $A(\hat{\mathcal{S}}) \supseteq \lim_{k\to \infty } \overline{\{(x_{j}(a),x_{j+1}(a))\}_{n\geq k}}$, what proves the equality because the closure of the orbit does not depends on the first point, only on the sequence $a \in \Omega$. The second part $\mathrm{proj}_{x} A(\hat{\mathcal{S}})=\lim_{k\to \infty } \overline{\{x_n (a)\}_{n\geq k}}$ follows in the same fashion using Theorem~\ref{EltErgTheoremGIFS}, or we can use the first part of this proof and the fact the projection is continuous.
\end{proof}

\subsection{Nonautonomous Dynamical Systems}

A well known example of Discrete Nonautonomous Dynamical System (see P\"otzsche \cite{MR2680867}) is the dynamics generated by a  finite difference equation(FDE), defined by a nonautonomous recurrence relation of order $m \geq 2$, nominally
$$x_{j+m}=f(x_{j+m-1}, x_{j+m-2}, ..., x_{j}, a_{j}), \; j\geq 0,$$
where $a_0, a_1, ... \in I$. The control sequence $(a_j)_{j=0,1, ...}$ represents some seasonal interference acting on each iteration by changing the standard  recursiveness on a FDE. It can be modeled by a GIFS $\mathcal{S}=(X^m, (\phi_j)_{j=0...n-1})$ where $\phi_{j}: X^m \to X$ is given by
$$\phi_{j}(y_{0}, y_{1}, ..., y_{m-1})=f(y_{0}, y_{1}, ..., y_{m-1}, a_{j}).$$
The control set $I$ could be finite or not. In this case, the orbit of the GIFS from the point  $(c_{0}, ..., c_{m-2}, c_{ m-1})$ is equal to the orbit of the FDE with initial conditions:
$$
\left\{
  \begin{array}{ll}
    x_{j+m} &=f(x_{j+m-1}, x_{j+m-2}, ..., x_{j}, a_{j}), a \in I^{\mathbb{N}}\\
    x_{0} & =c_0 \\
    \cdots & = \cdots \\
    x_{m-1}, & =c_{m-1}
  \end{array}
\right.
$$
If the associated GIFS satisfy the hypothesis E1 and E2, we can apply our theory to study the asymptotic behavior and the limit sets of these nonautonomous dynamical systems.

\begin{example} A bad example is $X=\mathbb{R}$, noncompact, and the not contractive GIFS  given by
$\phi_{-1}(x,y)= x - y \text{ and } \phi_{1}(x,y)=  x + y $ or $\phi_j (x,y)= x + j y $ associated to the FDE
$$
\left\{
  \begin{array}{ll}
    x_{j+2} &=f(x_{j+1}, x_{j}, a_{j})=x_{j+1}+ a_{j} x_{j} \\
    x_{0}, & =c_{0} \\
    x_{1}, & =c_{1}
  \end{array}
\right.
$$
Here,$I=\{-1,1\}$ and $\Omega=\{-1,1\}^{\mathbb{N}}$. In this case the theory does not work because there is no global attractor. However $A=\{(0,0)\}$ is an invariant set.
\end{example}

\begin{example}
A good example is the second order FDE in $X=[0,1]$ given by
$$
\left\{
  \begin{array}{ll}
    x_{j+2} &=\frac{1}{4} x_{j+1}+ \frac{1}{4} x_{j}+ \frac{a_j}{2},  a_j =j\in \{0,1\}\\
    x_{0}, & =c_{0} \in [0,1]\\
    x_{1}, & =c_{1} \in [0,1]
  \end{array}
\right.
$$
Associated to this FDE we have $\mathcal{S}=(X,  (\phi_j)_{j=0,1})$ an GIFS where $\phi_j (x,y)= \frac{1}{4}x + \frac{1}{4}y + \frac{j}{2}$.  Both functions $\phi_j $ are $Lip_{\frac{1}{4},\frac{1}{4}}([0,1]^2, [0,1])$  (this example appears in R. Miculescu~\cite{MR3180942}).

The typical limit set of the orbits $(x_n, x_{n+1})$, of this FDE is the attractor of the corresponding extended GIFS. In this case we can show explicitly a formula for those points:
$$x_{n+2}=\frac{b_{n+1}}{4^{n+1}}\,x_{{0}}+\frac{b_{n}}{4^{n+1}}\,x_{1}+ \frac{b_{n}}{
4^{n} \, 2} \, a_{{0}}+ \cdots +\frac{b_{0}}{ 4^{0}\,2}a_{n}, \; n\geq 0,$$
where $x_{0} = Z_{0, a} (x_0, x_1)$, $x_{1}= Z_{1, a} (x_0, x_1) $ and  $x_{j} = Z_{j, a} (x_0, x_1)$ satisfies $x_{j+1} =  \phi_{a_{j-1}}(x_{j-1}, x_{j}), \; j \geq 1.$ And, the numbers $b_{{0}}, b_{{1}}, ... $ are given by the generating function $G(w)={\frac{1}{1-w-4\,{w}^{2}}}=1+w+5\,{w}^{2}+9\,{w}^{3}+29\,{w}^{4}+65\,{w}^{5}+181\,{w}^{6}+441\,{
w}^{7}+1165\,{w}^{8}+O \left( {w}^{9} \right).$  Solving $1-w-4\,{w}^{2}=0$ with respect  to $z$ we get
$b_n=\frac{1}{\sqrt{17}}\left[\lambda_2 ^{-n-1} - \lambda_1 ^{-n-1}\right], \; n \geq 2.$
Finally,  the coefficients of $x_{{0}}$ and $x_{{1}}$ are of the type
$$\frac {b_{n-1}}{2^{2n+2}}= \frac {1}{2^{2n+2}\sqrt{17}} \left[\lambda_2 ^{-(n-1)-1} - \lambda_1 ^{-(n-1)-1}\right]= \frac{1}{4\sqrt{17}}\left[  \frac {\lambda_2^{-1}}{(4\lambda_2) ^{n }}  -    \frac {\lambda_1^{-1}}{(4\lambda_1) ^{n }}\right]  \to 0,$$
because $\|4\lambda_i\|>1$, $i=1,2$$(4\lambda_1 \approx -2.56...$ and $4\lambda_2 \approx 1.56...)$.

By the Chaos Game Theorem, the closure in $X$ of the orbit defined by $a=(a_0, a_1, ...)$ is $\mathrm{proj}_{x} A(\hat{\mathcal{S}})$:
$$\displaystyle\mathrm{proj}_{x} A(\hat{\mathcal{S}})= \lim_{k \to \infty}{\rm cl}\left( \bigcup_{n\geq k} \ x_{n}\right)=\lim_{k \to \infty}{\rm cl}\left( \bigcup_{n\geq k} \frac{1}{2} \sum_{j=0}^{n-2} {\frac {b_{j}}{
4^{j}}} a_{n-2-j}\right),$$ for a random $ a \in \Omega$. The  picture of $A(\hat{\mathcal{S}})$ is given  by the Figure~\ref{attracstdifs} and is drawn using the Chaos game theorem.
\begin{figure}[h!]
  \centering
  \includegraphics[width=6.5cm]{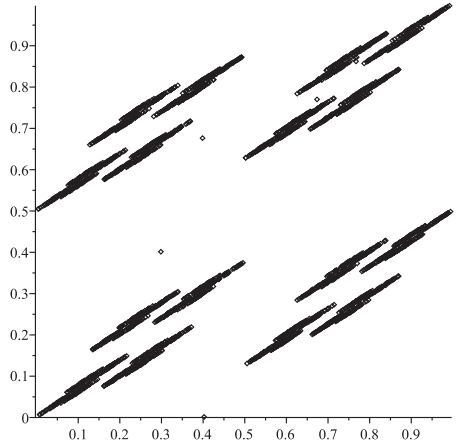}\\
  \caption{$\supp \alpha = A(\hat{\mathcal{S}})$, the attractor of $\hat{\mathcal{S}}$.
The picture of the attractor of extension is  obtained by 4000 iterations of a random orbit.}\label{attracstdifs}
\end{figure}

Employing the ergodic theorem we can also estimate the frequency of visitations of the solutions of the FDE on a subset of the phase space. We choose an initial distribution $p_0(x,y)$ and $p_1(x,y)$ satisfying hypothesis E2 (such as, $p_0(x,y)=\frac{1}{3}$ and $p_1(x,y)=\frac{2}{3}$). From Theorem~\ref{EltErgTheoremGIFS}, if $\mu_{\alpha}$ is the Hutchinson measure for $\mathcal{S}$ and $B \subseteq [0,1]$ there exists a measurable set $\Omega_{(c_0, c_1)}\subseteq \Omega$ such that $P_{(c_0, c_1)}(\Omega_{(c_0, c_1)})=1$ and  for any $a \in \Omega_{(c_0, c_1)}$
$$\frac{1}{N} \sharp\left\{ _{ 0\leq j\leq N-1} \; | \; x_{j}(a) \in B \right\} \to \mu_{\alpha}(B).$$
That is the typical average of visits of the orbits on the set $B$, with respect to probabilities that we choose.
\end{example}

\subsection{An example from Thermodynamic Formalism} \label{An example from Thermodynamic Formalism}
In Thermodynamic Formalism (see  Bowen~\cite{MR2423393} or Walters~\cite{MR0412389} for more details), the Ruelle theorem gives Gibbs measures for a potential with respect to a map.

\begin{theorem}\label{thermo form} (Ruelle)  Given, $(X,d)$ a compact and connected metric space, $T:X \to X$ an expanding map $n$ to $1$, $A:X \to \mathbb{R}$ a H\"older potential, there exists a positive number $\lambda$, a function $h$ and a measure $\nu$ such that $P_{A} h= \lambda h \text{ and } P_{A}^* \nu= \lambda \nu,$
where, $P_{A} (f)(x)= \sum_{Ty=x} e^{A(y)} f(y)$ is the Ruelle operator and his dual $P_{A}^*$, act in probabilities. Moreover, $\mu=h \nu$ called the Gibbs measure, is $T$-invariant. In particular the Gibbs property implies that  $\mu$ positive in open sets.
\end{theorem}

Usually, such dynamical systems are identified as a uniform IFSpdp $$\mathcal{R}=(X, (\tau_{i}(x))_{i=1,2,..,n}, (p_{i}(x)=e^{A(\tau_{i}(x))})_{i=1,2,..,n})$$ where the inverse branches of $T$ are the contractive maps $\tau_{i}$, that is,  $T\circ\tau_{i}(x)=id(x), \; i=1,2,..,n$. In Jorgensen~\cite{MR2240643} this identification is called  \emph{the endomorphism case}. If the potential is normalized, that is, $\lambda=1$ then $h=1$ and the Gibbs measure $\mu=\mu_{\mathcal{R}}$ is the Hutchinson measure of the IFSpdp because
$$ P_{A} (f)(x)= \sum_{Ty=x} e^{A(y)} f(y)= \sum_{i=1..n} p_{i}(x) f(\tau_{i}(x)),$$
is the fractal operator associated to $\mathcal{R}$.

The easier case where this happens is for $X=[0,1]$, $T(x)=2x\mod 1$ and $A(x)= \ln \frac{1}{T'(x)}=-\ln 2$. In this case the inverse branches are $\tau_{i}(x)= \frac{1}{2} x + \frac{i}{2}, \; i=0,1$ and $p_{i}(x)=e^{-\ln 2}=\frac{1}{2}$. It is well know, that the Gibbs measure is the Lebesgue measure $dx$ in the interval (see Conze~\cite{MR1078079} for a detailed study of the operator $ P_{A}$ in this case).

The next example shows that $A(\hat{\mathcal{S}})=A(\mathcal{S})^2$ can happen. We consider a GIFS that ``duplicates"  the behavior of the IFS $\mathcal{R}$ built from $T(x)=2x\mod 1$.

\begin{example} \label{GIFSattrac2xmod1} The GIFSpdp $\mathcal{S}$ in $[0,1]$ given by $\phi_i (x,y)= \frac{1}{2} x + \frac{i}{2}, \; i=0,1$ and $p_i (x,y)= \frac{1}{2}, \; i=0,1$, satisfy:\\
a) $\mu_{\mathcal{S}}=dx$ and $A(\mathcal{S})=[0,1]$;\\
b) $\alpha=dx\,dy$ and $A(\hat{\mathcal{S}})=[0,1]^2$.
\begin{figure}[h!]
  \centering
  \includegraphics[width=5cm]{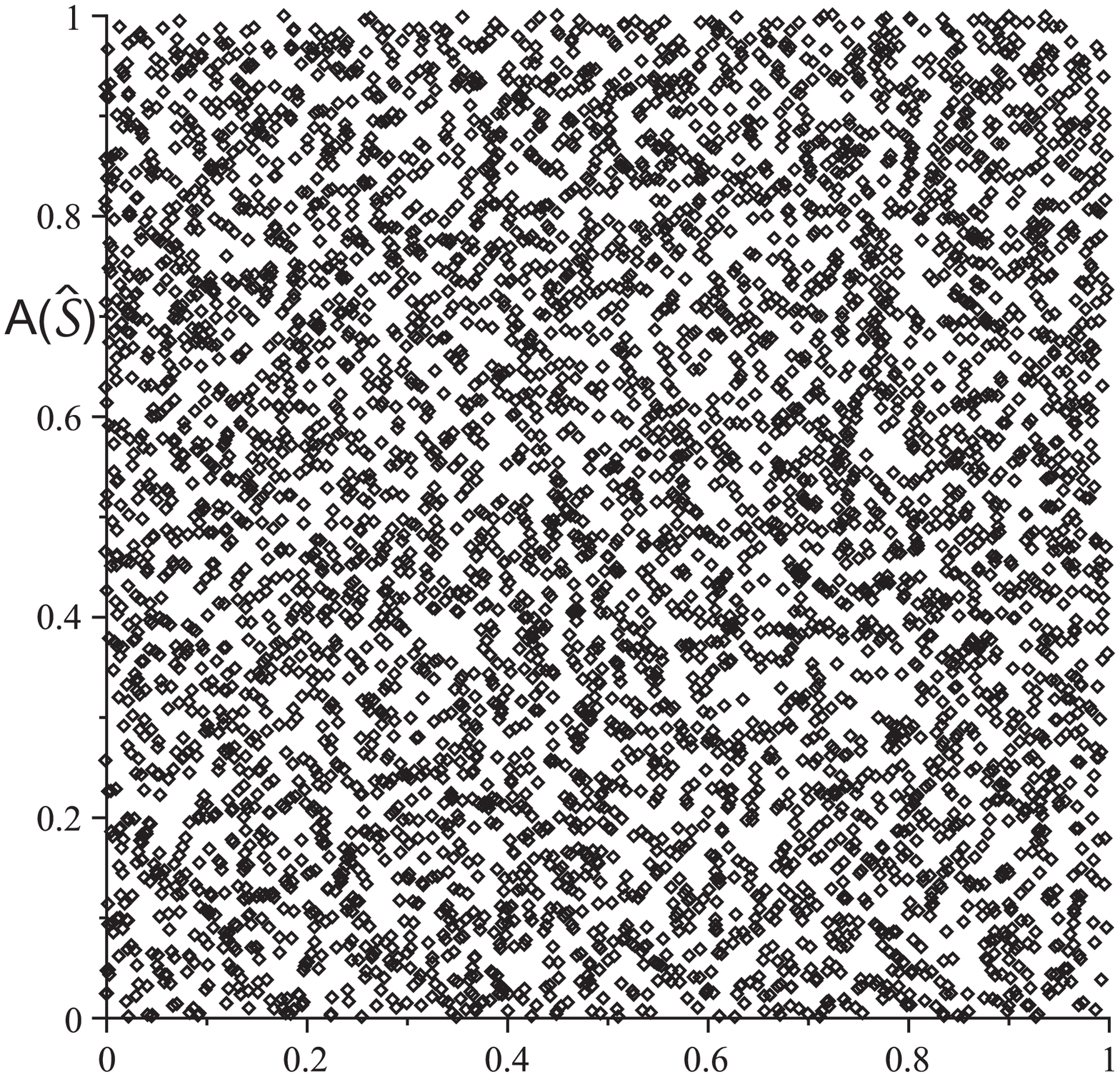}
  \includegraphics[width=5cm]{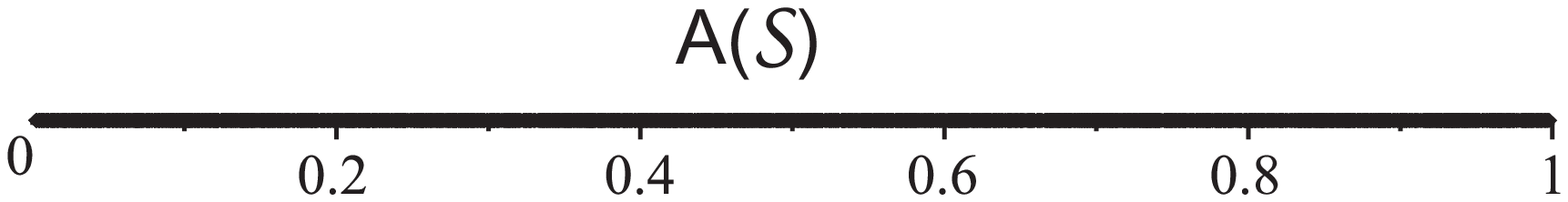}\\
  \caption{Chaos game, for 40000 iterations, showing that $A(\hat{\mathcal{S}})=A(\mathcal{S})^2$.}\label{attrac2xmod1}
\end{figure}

Indeed,
\begin{align*}
  \int f(x) \mathcal{L}_{\mathcal{S}}(dx,  dx) & =\int \int B_{\mathcal{S}}(f)(x,y)dx  \,dy \\
   & =\int \int  p_0 (x,y) f(\phi_0 (x,y)) + p_1 (x,y) f(\phi_1 (x,y)) dx  \, dy \\
   & =\int \int  \frac{1}{2} f\left(\frac{1}{2} x + \frac{0}{2}\right) + \frac{1}{2} f\left(\frac{1}{2} x + \frac{1}{2}\right) dx  \, dy \\
   & =\int \int  P_{-\ln 2} (f)(x) dx dy=\int dy \int  f(x) P_{-\ln 2}^*(dx) \\
   & = \int  f(x) dx,\\
\end{align*}
thus $\mathcal{L}_{\mathcal{S}}(dx, dx)=dx$, what means that the Hutchinson measure $\mu_{\mathcal{S}}$ is equal to the Lebesgue measure $dx$. Since $A(\mathcal{S})=\supp \mu_{\mathcal{S}}$ we get $A(\mathcal{S})=[0,1]$.

In order to prove the second claim, we consider the extended operator
\begin{align*}
  \int g(x,y) \mathcal{L}_{\hat{\mathcal{S}}}(dx \,dy) & =\int \int B_{\hat{\mathcal{S}}}(g)(x,y) dx \, dy \\
  & =\int \int  p_0 (x,y) g\left(\hat{\phi}_0 (x,y)\right) + p_1 (x,y) g\left(\hat{\phi}_1 (x,y)\right)  dx  \, dy \\
  & =\int \int  \frac{1}{2} g\left(y,\frac{1}{2} x + \frac{0}{2}\right) + \frac{1}{2} g\left(y, \frac{1}{2} x + \frac{1}{2}\right)  dx  \, dy \\
  & =\int \int  P_{-\ln 2} (g(y,\cdot))(x)  dx \, dy= \int \int  g(y,x)  P_{-\ln 2}^*(dx) dy \\
  & =\int \int  g(y,x)  dx \, dy=\int \int  g(x,y)  dx \, dy,
\end{align*}
thus $\mathcal{L}_{\hat{\mathcal{S}}}(dx  \,dy)=dx  \,dy$, what means that the extended Hutchinson measure $\alpha$ is equal to the Lebesgue measure $dx  \, dy$ in $[0,1]^2$. Since $A(\hat{\mathcal{S}})=\supp \alpha$ we get $A(\hat{\mathcal{S}})=[0,1]^2$.
\end{example}

\addcontentsline{toc}{section}{References}
\bibliographystyle{plain}

\end{document}